\documentclass[journal]{IEEEtran}
\usepackage{amssymb}
\usepackage[pdftex]{graphicx}
\usepackage{tikz}

\usepackage{framed}
\usepackage{bm}
\usepackage{cite}
\usepackage{comment}
\usepackage{algorithmic}
\usepackage{algorithm}
\usepackage{enumerate}
\usepackage{amsmath}
\usepackage{amsfonts}
\usepackage{color}
\usepackage{cases}

\usetikzlibrary{intersections,calc,arrows.meta}

\newtheorem{theorem}{Theorem}
\newtheorem{assumption}{Assumption}

\newtheorem{remark}{Remark}
\newtheorem{lemma}{Lemma}

\newtheorem{proposition}{Proposition}

\newtheorem{problem}{Problem}

\newtheorem{proof}{Proof}

\usepackage{blkarray}

\usepackage{cleveref}
\usepackage{mathtools}
\usepackage{subcaption}
\usepackage{multirow}
\usepackage{threeparttable}

\usepackage{pgfplots}
\pgfplotsset{compat=1.18}
\usepgfplotslibrary{statistics}

\crefname{section}{Section}{Sections}
\crefname{theorem}{Theorem}{Theorems}
\crefname{proposition}{Proposition}{Propositions}
\crefname{lemma}{Lemma}{Lemmas}
\crefname{problem}{Problem}{Problems}
\crefname{assumption}{Assumption}{Assumptions}
\crefname{algorithm}{Algorithm}{Algorithms}
\crefname{remark}{Remark}{Remarks}
\crefname{equation}{}{}
\Crefname{equation}{Equation}{Equations}
\crefname{figure}{Fig.}{Figs.}
\crefname{table}{Table}{Tables}

\newcommand{\RR}{\mathbb{R}}

\newcommand{\NN}{\mathbb{N}}
\newcommand{\norm}[1]{\left\Vert{#1}\right\Vert}
\newcommand{\rd}{\mathrm{d}}
\DeclareMathOperator*{\minimize}{minimize}
\DeclareMathOperator*{\argmin}{arg\,min}

\newcommand{\algorithmicbreak}{\textbf{break}}
\newcommand{\BREAK}{\STATE \algorithmicbreak}

\crefname{ALC@unique}{Step}{Steps}

\usepackage{latexsym}
\def\qed{\hfill $\Box$} 

\usepackage{cite}
\usepackage{amsmath}
\usepackage{url}

\begin{document}
%
\title{Controllability scores of linear time-varying network systems}
\author{Kota Umezu and Kazuhiro Sato\thanks{K. Umezu and K. Sato are with the Department of Mathematical Informatics, Graduate School of Information Science and Technology, The University of Tokyo, Tokyo 113-8656, Japan, email: krr0814bz@g.ecc.u-tokyo.ac.jp (K. Umezu), kazuhiro@mist.i.u-tokyo.ac.jp (K. Sato) }}
\maketitle
\thispagestyle{empty}
\pagestyle{empty}

\begin{abstract}
For large-scale network systems, network centrality based on control theory plays a crucial role in understanding their properties and controlling them efficiently.
The controllability score is such a centrality index and can give a physically meaningful measure.
It is originally proposed for linear time-invariant (LTI) systems, and we extend it to linear time-varying (LTV) systems in this paper.
Since the controllability score is defined as an optimal solution to some optimization problem, it is not necessarily uniquely determined.
Its uniqueness must be guaranteed for reproducibility and interpretability.
In this paper, we show its uniqueness in almost all cases, which guarantees its use as a network centrality measure.
We also prove its continuity with respect to the time parameters.
In addition, we propose a data-driven method to compute it.
Finally, we verify the effectiveness of the extension and examine the performance of the data-driven method through numerical experiments.
\end{abstract}


\IEEEpeerreviewmaketitle

\section{Introduction}
\label{sec: introduction}

\subsection{Background}
\label{subsec:background}

Large-scale dynamical systems on networks are ubiquitous across various fields: power grids~\cite{FuchsMorari2013} and multi-agent systems~\cite{FitchLeonard2016} in engineering, brain networks~\cite{Gu2015}, cellular reprogramming~\cite{Ronquist2017}, and ecosystems~\cite{Zhang2020} in natural sciences, and opinion networks~\cite{ChenYong2021} in social sciences.
Although nonengineering network systems are not necessarily controlled artificially, like engineering systems, they alter their dynamics in response to input signals and thus fall within the scope of modern control theory.
For instance, the brain alters its dynamics for task demands~\cite{Gu2015}, and ecosystems shift their dynamics in response to external disturbances~\cite{Zhang2020}.
Therefore, studying these networks from the perspective of modern control theory is crucial for performing efficient control or uncovering system properties, whether for engineering or nonengineering network systems, and has become an active area of research~\cite{Souza2023}.

One approach to analyzing large-scale network systems is to identify key nodes in their dynamics.
Among such approaches, a prominent method is the one proposed in \cite{Liu2011}, which is based on structural controllability~\cite{Lin1974}, a qualitative concept.
The method utilizes graph-theoretic tools and identifies a minimum set of nodes with which structural controllability is ensured when signals are applied.
The selected nodes are considered key nodes for control.
However, structural controllability is not necessarily a physically meaningful concept since structurally controllable systems may require a vast amount of energy for control~\cite{Gang2015}, and thus controlling them is sometimes unrealistic.
Consequently, as pointed out in \cite{Pasqualetti2014}, quantitative approaches are more favorable.

One of the prominent quantitative approaches is discussed in \cite{Summers2016}.
The method also identifies a set of key nodes by solving a combinatorial optimization problem based on a quantitative controllability metric.
Alternatively, assessing network centrality quantitatively is also a frequently utilized way~\cite{Summers2016,Pasqualetti2014}.
For instance, ranking by centrality measure based on quantitative controllability is applied to brain networks~\cite{Gu2015}.
The advantage of network centrality is that it provides a relative measure of importance rather than a binary assessment of whether being a key node or not.
The controllability score is such a centrality index, and some numerical experiments show that it provides a more reasonable measure than other existing indices.

The controllability score was originally developed for linear time-invariant (LTI) systems~\cite{SatoTerasaki2024}:
\begin{equation}\label{eq:lti_base}
    \dot{x}(t)=Ax(t).
\end{equation}
As explained in more detail in \cref{subsec:settings}, the optimal solution to a specific optimization problem can be interpreted as the importance of each node in the network; the controllability score is thus defined as the optimal solution.
If the optimal solution is not unique, however, different researchers employing the controllability score for the same system may reach inconsistent conclusions, thereby raising concerns about reproducibility.
Hence, the practical utility of the controllability score as a network centrality requires that the optimal solution be unique, and the condition under which the controllability score for LTI systems is unique has been investigated in \cite{SatoTerasaki2024,SatoKawamura2025}.

However, there is a problem that LTI systems cannot capture dynamics on a network whose structure varies over time, such as temporal networks~\cite{HolmeSaramaki2012}, unlike linear time-varying (LTV) systems:
\begin{equation}\label{eq:ltv_base}
    \dot{x}(t)=A(t)x(t).
\end{equation}
Recent research reports that LTV systems on temporal networks and LTI systems exhibit qualitatively and quantitatively different properties, such as the minimum energy for control~\cite{Li2017} and driver nodes~\cite{Srighakollapu2022}.
Therefore, it was necessary to extend the controllability score so that it can be applied to LTV systems as well.

One such extension was presented in \cite{Mo2025}, which addressed a switched system, i.e., a system whose matrix $A(t)$ is piecewise constant.
On each time interval where $A(t)$ remains constant, the system can be regarded as LTI.
The proposed one, which is called the generalized controllability score (GCS), evaluates the importance of each node on these intervals and is essentially identical to the one for LTI systems.
Consequently, the result of the uniqueness analysis for LTI systems can be employed.

However, a limitation of that approach is the requirement to know the switching times at which the matrix $A(t)$ changes in advance.
Although methods to identify switching times even when the system matrix is unknown have been proposed~\cite{Hakem2016}, they require the system matrix to be stable and therefore cannot be applied in general settings.

In this paper, we propose a new extension that is applicable to more general LTV systems, which we also refer to as the controllability score.
The proposed one does not assess the importance of the nodes on each subinterval; instead, it computes the importance over the entire interval.
As a result, even for switched systems, it requires no identification of the switching times.
As in previous studies~\cite{SatoTerasaki2024,SatoKawamura2025}, the interval under analysis is to be chosen by the analyst according to the intended application.

The temporal network model employed in \cite{Li2017}, which is identical to the switched system model considered in \cite{Mo2025}, exhibits solely abrupt changes, i.e., a sudden switching of the network topology as illustrated in \cref{fig:abrupt_change}.
However, gradual changes, in which edge weights vary continuously while the network topology remains fixed (see \cref{fig:gradual_change}), are also often of central importance, for instance in gene-regulatory networks~\cite{Wit2015} and in the detection of pandemics in social networks~\cite{Yamanishi2021}.
Accordingly, as detailed in \cref{subsec:systems}, we handle a model that can capture both abrupt and gradual changes, which is more general than the temporal network model.

\begin{figure}[htbp]
    \centering
    \begin{subfigure}{1.0\linewidth}
        \centering
        \begin{tikzpicture}[scale=0.8]
            \begin{scope}[shift={(0,0)}]
                \node (v1) at (0,0) {1};
                \node (v2) at (2,0) {2};
                \node (v3) at (1,1.6) {3};
                \node (v4) at (1,-1.6) {4};
                
                \draw[thick] (v1) circle[radius=0.30];
                \draw[thick] (v2) circle[radius=0.30];
                \draw[thick] (v3) circle[radius=0.30];
                \draw[thick] (v4) circle[radius=0.30];

                \draw[-Latex,thick] (v1) -- (v2);
                \draw[-Latex,thick] (v1) -- (v3);
                \draw[-Latex,thick] (v3) -- (v2);
                \draw[-Latex,thick] (v2) -- (v4);
            \end{scope}
            
            \draw[->] (3.2,0) -- (4.7,0);
            
            \begin{scope}[shift={(5.5,0)}]
                \node (w1) at (0,0) {1};
                \node (w2) at (2,0) {2};
                \node (w3) at (1,1.6) {3};
                \node (w4) at (1,-1.6) {4};
                
                \draw[thick] (w1) circle[radius=0.30];
                \draw[thick] (w2) circle[radius=0.30];
                \draw[thick] (w3) circle[radius=0.30];
                \draw[thick] (w4) circle[radius=0.30];
                
                \draw[-Latex,thick] (w2) -- (w1);
                \draw[-Latex,thick] (w1) -- (w3);
                \draw[-Latex,thick] (w1) -- (w4);
                \draw[-Latex,thick] (w4) -- (w3);
            \end{scope}
        \end{tikzpicture}
        \subcaption{abrupt change, representing a sudden switching of the network topology}
        \label{fig:abrupt_change}
    \end{subfigure}

    \begin{subfigure}[b]{1.0\linewidth}
        \centering
        \begin{tikzpicture}[scale=0.8]
            \begin{scope}[shift={(0,0)}]
                \node (v1) at (0,0) {1};
                \node (v2) at (2,0) {2};
                \node (v3) at (1,1.6) {3};
                \node (v4) at (1,-1.6) {4};
                
                \draw[thick] (v1) circle[radius=0.30];
                \draw[thick] (v2) circle[radius=0.30];
                \draw[thick] (v3) circle[radius=0.30];
                \draw[thick] (v4) circle[radius=0.30];

                \draw[-Latex,thick] (v1) -- (v2);
                \draw[-Latex,thick] (v1) -- (v3);
                \draw[-Latex,thick] (v3) -- (v2);
                \draw[-Latex,thick] (v2) -- (v4);
            \end{scope}
            
            \draw[->] (3.2,0) -- (4.7,0);
            
            \begin{scope}[shift={(5.5,0)}]
                \node (w1) at (0,0) {1};
                \node (w2) at (2,0) {2};
                \node (w3) at (1,1.6) {3};
                \node (w4) at (1,-1.6) {4};
                
                \draw[thick] (w1) circle[radius=0.30];
                \draw[thick] (w2) circle[radius=0.30];
                \draw[thick] (w3) circle[radius=0.30];
                \draw[thick] (w4) circle[radius=0.30];

                \draw[-Latex,ultra thick] (w1) -- (w2);
                \draw[-Latex,dashed] (w1) -- (w3);
                \draw[-Latex,thick] (w3) -- (w2);
                \draw[-Latex,dashed] (w2) -- (w4);
            \end{scope}
        \end{tikzpicture}
        \subcaption{gradual change, representing continuous variation of edge weights under a fixed topology}
        \label{fig:gradual_change}
    \end{subfigure}
    \caption{Illustrations of the two types of network evolution}
    \label{fig:change}
\end{figure}

\subsection{Challenge}
\label{subsec:challenge}

\begin{itemize}
    \item A significant challenge is the requirement for knowledge of the system.
    Since the controllability score quantitatively evaluates centrality, it requires certain numerical information about the system, for instance, the numerical values of the system matrix.
    Nevertheless, system identification is sometimes difficult, especially for LTV systems~\cref{eq:ltv_base}, unless assumptions such as sparsity are imposed~\cite{Mei2016}, since the system matrix $A(t)$ changes over time.

    \item Since the GCS~\cite{Mo2025} is essentially identical to the controllability score for LTI systems, the results for LTI systems can be directly utilized.
    In contrast, the controllability score extended in this paper deals with LTV systems as they are and therefore lies outside the scope of the results derived for LTI systems.

    \item The dependence of the controllability score on the time parameters was partially studied in \cite{SatoKawamura2025}.
    Specifically, for LTI systems, it was shown that when the time horizon is sufficiently small, the controllability score yields an approximately uniform evaluation.
    However, many aspects, such as continuity, remain unclear even for LTI systems.
\end{itemize}

\subsection{Contribution}
\label{subsec:contribution}

\begin{itemize}
    \item To broaden the applicability of the controllability score, we extend it in a different manner from \cite{Mo2025}.
    Moreover, as discussed in \cref{subsec:difference}, it is expected to capture the system dynamics in greater detail than the GCS~\cite{Mo2025}.
    Furthermore, we show that it is uniquely determined in almost all cases.
    The proof employs a technique similar to that in \cite{SatoKawamura2025}: we prove that the controllability score is unique for almost all time parameter values.
    Consequently, it can be considered practically unique and thus can be employed as a centrality measure.
    In numerical experiments, we demonstrate that the controllability score yields different scores for LTI and LTV systems, and that it can capture properties that existing indices fail to detect.
    These findings suggest the importance of the extension.

    \item As a step toward clarifying the dependence of the controllability score on the time parameters, we prove that it is continuous with respect to the time parameters.
    This result not only provides insight into how the controllability score depends on the time parameters but also ensures practical robustness, since small identification errors of switching times are expected to cause only minor deviations in the controllability score.
    In numerical experiments, we observe that small changes in the time parameters lead to relatively small variations in the controllability score, whereas large changes in the time parameters result in relatively large variations in the controllability score.

    \item We develop a data-driven method to compute the controllability score using experimental data instead of the system matrix.
    Specifically, the proposed method computes the controllability Gramian, which is required to calculate the controllability score.
    Although related work~\cite{Banno2021,Tsuji2024} has proposed data-driven methods for the controllability Gramian, the methods are limited to LTI systems since they employ the Lyapunov equation, which applies to LTI systems but not LTV systems.
    In contrast, our proposed method can be applied to LTV systems since it relies on the integral representation.
    Moreover, we emphasize that the advantage of our proposed method is its applicability not only to temporal network systems, where the system matrix switches at certain times but remains piecewise constant, but also to general systems whose system matrices change continuously over time.
    Systems that are suitably modeled by the former are somewhat easier to identify~\cite{Wang2024}.
    However, for systems where the latter is appropriate, identification is challenging as noted in \cref{subsec:challenge}, highlighting the advantage of this method.
    Furthermore, the GCS~\cite{Mo2025} requires at least the identification of the switching times, whereas the proposed method does not require such identification.
    Numerical experiments show that our proposed method can accurately compute the controllability Gramian when the number of observations is sufficiently large.
\end{itemize}

\subsection{Outline}
\label{subsec:outline}

The rest of the paper is organized as follows.
In \cref{sec:preliminaries}, we introduce notations and summarize essential concepts such as temporal networks and the controllability Gramian.
In \cref{sec:ltv_score}, we extend the controllability score to apply to LTV systems and clarify the differences from \cite{Mo2025}.
We prove its uniqueness for almost all time parameter values and its continuity with respect to the time parameters.
In \cref{sec:algorithm}, we summarize the optimization algorithm to compute the controllability score and propose a data-driven method.
In \cref{sec:experiments}, we show numerical experiments to compare controllability scores with existing centrality indices, to compare them between LTI and LTV systems, to investigate their dependence on the time parameters and the chronological order of the snapshots, and to assess the performance of the proposed data-driven method.
The concluding remarks are given in \cref{sec:conclusion}.

\section{Preliminaries}
\label{sec:preliminaries}

\subsection{Notation}
\label{subsec:notation}

The set of all real numbers and the set of all natural numbers are denoted by $\RR$ and $\NN$, respectively.
Let $n$ denote the number of nodes, and let $I$ and $O$ denote the identity matrix of order $n$ and the $n\times n$ zero matrix, respectively.
Let $e_i\coloneq (0,\ldots,0,1,0,\ldots,0)^\top\in\RR^n$ be a standard vector whose $i$-th element is $1$ and other elements are $0$.
For a vector $v=(v_1,\ldots,v_k)^\top\in\RR^k$, $\norm{v}\coloneq\sqrt{\sum_{i=1}^k v_i^2}$ denotes the standard Euclidean norm.
The symbol $L^2([a,b];\RR^k)$ denotes the set of all square-integrable functions $u\colon [a,b]\to\RR^k$, i.e., $L^2([a,b];\RR^k)\coloneq \left\{u\colon [a,b]\to\RR^k\mid \int_{a}^{b} \norm{u(t)}^2\rd t<\infty \right\}$.
For a square-integrable function $u\in L^2([a,b];\RR^k)$, $\norm{u}_{L^2}\coloneq \sqrt{\int_{a}^{b} \norm{u(t)}^2\rd t}$ denotes the $L^2$ norm.

For a matrix $A\in\RR^{n\times n}$, $\mathrm{e}^A$, $\det A$, and $\mathrm{tr}(A)$ denote the exponential of $A$, the determinant of $A$, and the trace of $A$, respectively.
When $A$ is symmetric, we write $A \succ O$ to mean that $A$ is positive definite.
For $v_1,\ldots,v_n\in\RR$, let $\mathrm{diag}(v_1,\ldots,v_n)$ denote the diagonal matrix with the diagonal elements $v_1,\ldots,v_n$.
Let $\Delta$ be a standard simplex in $\RR^n$, i.e., $\Delta \coloneq \left\{p\in\RR^n \;\middle|\; p_i\geq0\ (i=1,\ldots,n),\ \sum_{i=1}^n p_i=1\right\}$.
A function $\varphi\colon D\to\RR$ is said to be real analytic if, for every point $\overline{x}\in D$, there exists an open neighborhood $V\subset D$ of $\overline{x}$ such that the Taylor series
\begin{equation*}
    \sum_{\alpha\in\NN^k} \dfrac{1}{\alpha!}\partial^{\alpha}\varphi(\overline{x})(x-\overline{x})^{\alpha}
\end{equation*}
converges absolutely on $V$ and coincides with $\varphi(x)$, where its domain $D\subset\RR^k$ is open, $\alpha=(\alpha_1,\ldots,\alpha_k)$ is a multi-index, $\alpha!=\alpha_1!\cdots\alpha_k!$, $\partial^{\alpha}=\partial_{x_1}^{\alpha_1}\cdots\partial_{x_k}^{\alpha_k}$, and $(x-\overline{x})^{\alpha}=(x_1-\overline{x}_1)^{\alpha_1}\cdots(x_k-\overline{x}_k)^{\alpha_k}$.

\subsection{Dynamical systems}
\label{subsec:systems}

In this paper, we consider an LTV system~\cref{eq:ltv_base} in which the matrix $A(t)$ is expressed as
\begin{multline}\label{eq:ltv_matrix}
    A(t;\Delta t)= \\
    \begin{cases}
        A_k(t-t_{k-1})& \text{if}\ t_{k-1}\leq t<t_k\ (k=1,\ldots,m-1), \\
        A_m(t-t_{m-1})& \text{if}\ t_{m-1}\leq t\leq t_m,
    \end{cases}
\end{multline}
where $A_k(t)$ is an $n\times n$ real matrix-valued function that is real analytic on $(0,s_k)$ and continuous on $[0,s_k]$ for some $s_k>0$, $\Delta t=(\Delta t_1,\ldots,\Delta t_m) \ (0\leq \Delta t_k\leq s_k, \ k=1,\ldots,m)$ is a tuple of durations, and $t_0=0,\ t_k=\sum_{\ell=1}^k \Delta t_{\ell}$.
Since a real-analytic function cannot always be extended analytically over the whole interval but only within a bounded region, $\Delta t_k$ cannot take arbitrary positive values; its admissible range may be restricted depending on the form of the function $A_k(t)$.
The constant $s_k>0$ is introduced to define admissible values of $\Delta t_k$, and it is required in \cref{prop:uniqueness} to describe the set of $\Delta t$ for which \cref{prob:vcs,prob:aecs} each admit a unique optimal solution.
However, $s_k$ is used solely to ensure the theoretical rigor of the statements; in practice, there is no need to specify $s_k$ explicitly.

In the case $m=1$ and $A_1(t)\equiv A$, the system reduces to an LTI system \cref{eq:lti_base} on the time interval $[0,\Delta t_1]$.
The system can capture gradual changes in the network structure via the time evolution of each $A_k(t)$ since any continuous function can be approximated arbitrarily well by real-analytic functions, while it can also capture abrupt changes via the possibly discontinuous jump from $A_k(\Delta t_k)$ to $A_{k+1}(0)$ at $t=t_k$ (see \cref{fig:abrupt_change}).
The temporal network model employed in \cite{Li2017}, which is identical to the switched system model employed in \cite{Mo2025}, corresponds to the case where $A_k(t)$ is constant; it accounts for abrupt changes but does not incorporate gradual changes.
We also refer to such a system as a temporal network or a switched system in this paper.

For an LTV system with \cref{eq:ltv_matrix}, $\Phi(t,\tau;\Delta t)$ denotes its state transition matrix~\cite[Definition 4.2]{Chen1999}, which is continuous and satisfies
\begin{equation}\label{eq:transition_definition}
    \dfrac{\partial}{\partial t}\Phi(t,\tau;\Delta t)=A(t;\Delta t)\Phi(t,\tau;\Delta t),\quad \Phi(\tau,\tau;\Delta t)=I.
\end{equation}
Here, we consider \cref{eq:transition_definition} in the sense of one-sided differentials at $t_k\ (k=0,\ldots,m)$.
When $\tau\leq t, t_{k-1}\leq t\leq t_k, t_{\ell-1}\leq \tau\leq t_{\ell}$, from the properties of the state transition matrix,
\begin{multline*}
    \Phi(t,\tau;\Delta t)=\Phi(t,t_{k-1};\Delta t)\Phi(t_{k-1},t_{k-2};\Delta t) \\
    \cdots\Phi(t_{\ell+1},t_{\ell};\Delta t)\Phi(t_{\ell},\tau;\Delta t)
\end{multline*}
holds.
Moreover, let $\Phi_k(t',\tau')$ denote the state transition matrix of the LTV system corresponding to $A_k(t')$.
For $t_{k-1}\leq t'< t$, it follows that $A(t';\Delta t)=A_k(t'-t_{k-1})$; hence, $\Phi(t,t_{k-1};\Delta t)=\Phi_k(t-t_{k-1},0)$.
Similarly, relations such as $\Phi(t_{k-1},t_{k-2};\Delta t)=\Phi_{k-1}(\Delta t_{k-1},0)$ and $\Phi(t_{\ell},\tau;\Delta t)=\Phi_{\ell}(\Delta t_\ell,\tau-t_{\ell-1})$ also hold.
Therefore, we obtain
\begin{equation}\label{eq:transition}
    \begin{split}
        \Phi(t,\tau;\Delta t)&=\Phi_k(t-t_{k-1},0)\Phi_{k-1}(\Delta t_{k-1},0) \\
        &\quad\cdots\Phi_{\ell+1}(\Delta t_{\ell+1},0)\Phi_{\ell}(\Delta t_\ell,\tau-t_{\ell-1}),
    \end{split}
\end{equation}
where $\tau\leq t, t_{k-1}\leq t\leq t_k, t_{\ell-1}\leq \tau\leq t_{\ell}$.
This representation will be employed in the proofs of \cref{prop:analyticity,prop:uniqueness}.

\subsection{Controllability Gramian and minimum-energy control}
\label{subsec:gramian}

In this subsection, we quickly review the controllability Gramian and then summarize the results of minimum-energy control.

We consider the following LTV system with control input:
\begin{equation}\label{eq:ltv_input}
    \dot{x}(t)=A'(t)x(t)+B(t)u(t),
\end{equation}
where $B(t)\in\RR^{n\times k}$ is the input matrix and $u(t)\in\RR^k$ is the control input.
Let $\Phi'(t,\tau)$ denote the state transition matrix corresponding to $A'(t)$.
The finite-time controllability Gramian for \cref{eq:ltv_input} is defined as
\begin{equation*}
    \mathcal{W}(T)\coloneq\int_0^T \Phi'(T,\tau) B(\tau)B(\tau)^\top\Phi'(T,\tau)^\top\rd\tau,
\end{equation*}
which is a symmetric positive semidefinite matrix.
It is well-known that this matrix is related to the controllability of the system \cref{eq:ltv_input}.
The system is, for instance, controllable on the time interval $[0,T]$ if and only if $\mathcal{W}(T)\succ O$~\cite{Kalman1963}.

Here, controllability itself is a concept that focuses solely on whether some control input can steer the state vector from the origin to any desired state within a finite time.
The magnitude of the energy required for such a control input is not considered.
That is, even if a system is controllable, achieving the desired state may be unrealistic since the required energy may be too large~\cite{Gang2015}.
Thus, assessing the ability to control quantitatively is crucial when controlling a system.
In this paper, we also refer to the quantitative control ability as controllability.

The controllability Gramian is also related to quantitative controllability, and the following result is known.
\begin{proposition}[minimum energy for control~\cite{Kalman1963}]\label{prop:ltv_energy}
    Assume that the LTV system \cref{eq:ltv_input} is controllable on the time interval $[0,T]$, which is equivalent to $\mathcal{W}(T)\succ O$.
    Then, for any desired state vector $x_{\mathrm{f}}\in\RR^{n}$, the minimum energy required for driving the state vector from the origin at time $0$ to $x_{\mathrm{f}}$ at time $T$ is given by
    \begin{align*}
        &\min_{u\in L^2([0,T];\RR^k)} \left\{\norm{u}_{L^2}^2 \;\middle|\; x(0)=0,\ x(T)=x_{\mathrm{f}}\ \textrm{under \cref{eq:ltv_input}}\right\} \\
        &=x_{\mathrm{f}}^\top \mathcal{W}(T)^{-1} x_{\mathrm{f}}.
    \end{align*}
\end{proposition}

Therefore, by employing the controllability Gramian, we can quantitatively evaluate the controllability of the LTV system \cref{eq:ltv_input}.
More specifically, we can utilize the following result~\cite[Proposition~4.6]{DullerudPaganini2000}, which follows from \cref{prop:ltv_energy}.
\begin{proposition}\label{prop:volume}
    Assume that the LTV system \cref{eq:ltv_input} is controllable on the time interval $[0,T]$.
    Then, the reachable space defined as
    \begin{align*}
        &\mathcal{E}(T)\coloneq \\
        &\left\{x_{\mathrm{f}}\in\RR^n \;\middle|\;
        \begin{gathered}
            \text{There exists $u\in L^2([0,T];\RR^k)$ s.t.} \\
            \text{$\norm{u}_{L^2}\leq 1,x(0)=0,x(T)=x_{\mathrm{f}}$ under \cref{eq:ltv_input}.}
        \end{gathered}
        \right\}
    \end{align*}
    can be expressed as
    \begin{equation*}
        \mathcal{E}(T)=\left\{y\in\RR^n \;\middle|\; y^\top \mathcal{W}(T)^{-1} y\leq 1\right\},
    \end{equation*}
    and, therefore, its volume is proportional to $\sqrt{\det \mathcal{W}(T)}$.
\end{proposition}

Without any constraint on the input energy, any state is reachable when a controllable system is considered; however, in practice, physical constraints exist.
We therefore consider the set of states that can be reached with inputs of energy at most one, namely the reachable set, as the set of states achievable under realistic constraints.
The larger this set is, the easier the system is to control.
Thus, the controllability of the system can be evaluated by the volume of the reachable set.
Consequently, we can use $\sqrt{\det \mathcal{W}(T)}$, or $\log \det \mathcal{W}(T)$, as a measure of controllability for LTV systems~\cref{eq:ltv_input}~\cite{MullerWeber1972}.
Since $\log\det \mathcal{W}(T)$ can often be computed more stably than $\sqrt{\det \mathcal{W}(T)}$ and its concavity is useful for optimization, we use $\log\det \mathcal{W}(T)$ rather than $\sqrt{\det \mathcal{W}(T)}$.

Alternatively, we can also utilize the following result~\cite{Kalman1963}.

\begin{proposition}\label{prop:average_energy}
    Assume that the LTV system \cref{eq:ltv_input} is controllable on the time interval $[0,T]$.
    Then, the average of the minimum energy required to drive the state vector from the origin to the point on the unit sphere $\{y\in\RR^n\mid \norm{y}=1\}$ over the uniform distribution is proportional to $\mathrm{tr}\left(\mathcal{W}(T)^{-1}\right)$.
\end{proposition}

Since physical constraints limit the input energy, the smaller the energy required to reach a desired state, the easier it is to realize that state.
If the desired state is given in advance, it suffices to consider only the energy required to reach it.
If not, however, a system that can reach a variety of states with low average energy should be regarded as more controllable.
Thus, the controllability of the system can be evaluated by the average of the minimum energy.
Consequently, we can regard $\mathrm{tr}\left(\mathcal{W}(T)^{-1}\right)$ as a measure of controllability for LTV systems~\cref{eq:ltv_input}~\cite{MullerWeber1972}.

LTV systems~\cref{eq:ltv_input} include LTI systems
\begin{equation}\label{eq:lti_input}
    \dot{x}(t)=Ax(t)+Bu(t)
\end{equation}
as a special case.
Therefore, \cref{prop:ltv_energy,prop:volume,prop:average_energy} also hold for LTI systems~\cref{eq:lti_input}.
For an LTI system~\cref{eq:lti_input}, the controllability Gramian can be expressed as follows:
\begin{equation*}
    \mathcal{W}(T)=\int_0^T \mathrm{e}^{\tau A}BB^\top \mathrm{e}^{\tau A^\top}\rd\tau.
\end{equation*}

\section{Controllability scores for LTV systems}
\label{sec:ltv_score}

In this section, we extend the concept of the controllability score, initially proposed for LTI systems on networks, to LTV systems on networks.
First, we formulate optimization problems and define the controllability score for LTV systems as its optimal solution in \cref{subsec:settings}.
Next, in \cref{subsec:difference}, we elaborate on the distinctions between ours and the one proposed in \cite{Mo2025}.
Then, we prove that the optimal solution is almost always unique in \cref{subsec:uniqueness}.
Finally, we examine the dependence of the controllability score on the time parameters.
Specifically, we show the continuity of the controllability score in \cref{subsec:continuity}.

\subsection{Formulation and definition}
\label{subsec:settings}

The controllability score~\cite{SatoTerasaki2024} is a network centrality measure that assesses the significance of each node in the dynamical network system.
Previous studies~\cite{SatoTerasaki2024,SatoKawamura2025} have focused on LTI system models \cref{eq:lti_base} of large-scale network systems.
In this section, we consider LTV system models \cref{eq:ltv_base} represented by $A(t;\Delta t)$ defined in \cref{subsec:systems}, and we extend the concept to apply to it in a different manner from \cite{Mo2025}.
Here, $x(t)=(x_1(t),\ldots,x_n(t))^\top\in\RR^n$ and $A(t;\Delta t)$ represent the states of the nodes and the time-varying structure of the network, respectively.

To define the controllability score, we consider the following equation, which adds a hypothetical control input term to \cref{eq:ltv_base}:
\begin{equation}\label{eq:ltv_diagonal}
    \dot{x}(t)=A(t;\Delta t)x(t)+\mathrm{diag}(\sqrt{p_1},\ldots,\sqrt{p_n})u(t),
\end{equation}
where $u(t)=(u_1(t),\ldots,u_n(t))^\top\in\RR^n$ is a hypothetical control input and $p=(p_1,\ldots,p_n)^\top\in\RR^n$ is assumed to satisfy
\begin{align}\label{eq:p_constraint}
        p_i\geq0\quad(i=1,\ldots,n), \quad 
        \sum_{i=1}^n p_i=1.
\end{align}
\Cref{eq:ltv_diagonal} corresponds to an LTV system \cref{eq:ltv_input} where $B(t)\equiv\mathrm{diag}(\sqrt{p_1},\ldots,\sqrt{p_n})$.
Note that $p$ is not dependent on $t$, as detailed in \cref{subsec:difference}.
From \cref{eq:ltv_diagonal}, node $x_i$ and input $u_i$ correspond one-to-one.
The larger $p_i$ is, the more significant the influence of control input $u_i$ on node $x_i$ can be.

Here, let us regard $p$ as a design variable and consider maximizing the controllability of the system \cref{eq:ltv_diagonal} with respect to some measure.
When the optimal solution is $p^*=(p_1^*,\ldots,p_n^*)^\top$, if $p_i^*$ is large, it indicates that we can efficiently control the system \cref{eq:ltv_diagonal} by actively influencing node $x_i$.
Thus, we can consider node $x_i$ as a pivotal node.
On the other hand, if $p_i^*$ is small, it suggests that we can control the system \cref{eq:ltv_diagonal} without significantly influencing node $x_i$, meaning that node $x_i$ is not a critical node.
Therefore, the optimal solution $p^*$ can be interpreted as the importance of each node, and the constraint \cref{eq:p_constraint} represents the importance distribution, where it is nonnegative and the sum equals $1$.
The controllability score is defined as the optimal solution $p^*$.

Let
\begin{equation}\label{eq:p_gramian}
    W(p;\Delta t)\coloneq\int_{0}^{t_m}\Phi(t_m,\tau;\Delta t)\left\{\mathrm{diag}(p)\right\}\Phi(t_m,\tau;\Delta t)^\top\rd\tau
\end{equation}
denote the controllability Gramian for the considered LTV system \cref{eq:ltv_diagonal}.
As described in \cref{subsec:gramian}, several indices for controllability are possible. Accordingly, let us consider the following two optimization problems.
\begin{problem}\label{prob:vcs}
    \begin{align*}
        \minimize_{p} &\quad -\log\det W(p;\Delta t) \\
        \mathrm{subject\ to} &\quad p\in\Delta,\ W(p;\Delta t)\succ O.
    \end{align*}
\end{problem}
\begin{problem}\label{prob:aecs}
    \begin{align*}
        \minimize_{p} &\quad \mathrm{tr}\left(W(p;\Delta t)^{-1}\right) \\
        \mathrm{subject\ to} &\quad p\in\Delta,\ W(p;\Delta t)\succ O.
    \end{align*}
\end{problem}
Note that the constraint $p\in\Delta$ is equivalent to \cref{eq:p_constraint}.
The optimal solutions to \cref{prob:vcs,prob:aecs} are referred to as the volumetric controllability score (VCS) and the average energy controllability score (AECS), respectively.

However, two points should be noted when interpreting the controllability score as the importance of each node.
The first point is that an optimal solution must exist; however, this is not a significant issue since it can be easily proved in the same manner as the case of LTI systems.

\begin{theorem}\label{prop:existence}
    The optimal solutions to \cref{prob:vcs,prob:aecs} exist.
\end{theorem}
\begin{proof}
    See \cite[Theorem 2]{SatoTerasaki2024}. \qed
\end{proof}

The second point, a more significant issue, is that the optimal solution must be unique.
If it is not unique, a reproducibility problem arises since different researchers analyzing the same network may arrive at different conclusions.
In addition, from the perspective of interpretability, there is also an issue of determining node importance based on the solutions, which remains unclear.
For these reasons, the optimal solution must be unique.
However, there exists an LTI system with a specific time parameter value for which the optimal solutions to \cref{prob:vcs,prob:aecs} are not unique~\cite{SatoTerasaki2024}.
Therefore, clarifying the conditions under which controllability scores are uniquely determined is crucial for utilizing them as a centrality measure.
The following results for LTI systems have been presented regarding this issue.

\begin{proposition}[\cite{SatoTerasaki2024}]\label{prop:lti_unique_stable}
    Assume that $A$ in \cref{eq:lti_base} is stable, i.e., each eigenvalue has a negative real part.
    Then, for any duration $\Delta t_1>0$, \cref{prob:vcs,prob:aecs} each admit a unique optimal solution.
\end{proposition}

\begin{proposition}[\cite{SatoKawamura2025}]\label{prop:lti_unique_any}
    Assume that $A$ in \cref{eq:lti_base} is arbitrary.
    Then, for almost all durations $\Delta t_1>0$, \cref{prob:vcs,prob:aecs} each admit a unique optimal solution.
\end{proposition}

The uniqueness in the case of almost all durations is sufficient for practical purposes.
Thus, from \cref{prop:lti_unique_any}, the use of the controllability scores poses no practical problems when considering LTI systems \cref{eq:lti_base} on networks.
Similarly, for LTV systems, we prove that the optimal solution is unique for almost all tuples of durations $\Delta t$ in \cref{subsec:uniqueness}.

The dependence of the controllability score on the time parameters has so far been examined only partially~\cite{SatoKawamura2025}.
Accordingly, in \cref{subsec:continuity}, we prove its continuity.

In addition, calculating the controllability scores is also an important issue.
A specific algorithm is described in \cref{sec:algorithm}.
The convexity of the objective function, which can be proved in the same manner as the case of LTI systems, is crucial for computation.
\begin{theorem}\label{prop:convexity}
    The objective functions of \cref{prob:vcs,prob:aecs} are convex.
\end{theorem}
\begin{proof}
    See \cite[Theorems 1 and 3]{SatoTerasaki2024}. \qed
\end{proof}

It has been shown that convexity allows efficient computation~\cite{SatoTerasaki2024,SatoKawamura2025}.

\subsection{Difference from the existing work}
\label{subsec:difference}

We explain the difference between ours and the GCS~\cite{Mo2025} and why we assume that $p$ is not dependent on $t$ in \cref{eq:ltv_diagonal}.
The GCS is designed for switched systems, where $A_k(t)\equiv A_k \ (k=1,\ldots,m)$.

The first major difference is that \cite{Mo2025} introduces a separate decision variable $p^{(k)}\in\RR^n$ for each time interval $[t_k,t_{k+1}]$, whereas our method employs a single variable $p\in\RR^n$ over the entire time interval $[0,t_m]$.
As described in \cref{subsec:settings}, $p$ is treated as a design variable, and the optimal solution to \cref{prob:vcs} or \ref{prob:aecs} is regarded as the importance of each node.
Thus, if a separate variable $p^{(k)}$ is introduced for each subinterval, or more generally if $p$ is assumed to depend on time $t$, the optimal solution represents the importance of nodes on each subinterval or at each time $t$.
Although it seems to contain rich information, the importance of the nodes over the entire interval remains unclear, and further analysis is necessary.
Additionally, in cases where the number of switching $m$ or the final time of the system $t_m$ is large, the computational cost required to calculate the importance for each node is enormous.

On the other hand, if we assume that $p$ does not depend on $t$, the optimal solution is also not dependent on $t$.
Thus, it can be considered a measure of importance reflecting the temporally global dynamics.
Therefore, we assume that $p$ does not depend on $t$.

\begin{figure}[htbp]
    \centering
    \begin{subfigure}{1.0\linewidth}
        \centering
        \begin{tikzpicture}[scale=0.8]
            \draw[->] (0,0) -- (9,0) node[below] {$t$};
            \draw (0.8,0.12) -- (0.8,-0.12) node[below] {$t_0$};
            \draw (4.5,0.12) -- (4.5,-0.12) node[below] {$t_1$};
            \draw (8.2,0.12) -- (8.2,-0.12) node[below] {$t_2$};

            \begin{scope}[shift={(2.65,2.0)}]
                \node (P1) at (90:1.3) {1};
                \node (P2) at (162:1.3) {2};
                \node (P3) at (234:1.3) {3};
                \node (P4) at (306:1.3) {4};
                \node (P5) at (18:1.3) {5};

                \draw[thick] (P1) circle[radius=0.30];
                \draw[thick] (P2) circle[radius=0.30];
                \draw[thick] (P3) circle[radius=0.30];
                \draw[thick] (P4) circle[radius=0.30];
                \draw[thick] (P5) circle[radius=0.30];

                \draw[-Latex,thick] (P1) -- (P2);
                \draw[-Latex,thick] (P3) -- (P1);
                \draw[-Latex,thick] (P2) -- (P3);
                \draw[-Latex,thick] (P4) -- (P2);
                \draw[-Latex,thick] (P3) -- (P5);
                \draw[-Latex,thick] (P5) -- (P4);
                \draw[-Latex,thick] (P1) -- (P5);
            \end{scope}

            \begin{scope}[shift={(6.35,2.0)}]
                \node (Q1) at (90:1.3) {1};
                \node (Q2) at (162:1.3) {2};
                \node (Q3) at (234:1.3) {3};
                \node (Q4) at (306:1.3) {4};
                \node (Q5) at (18:1.3) {5};
                
                \draw[thick] (Q1) circle[radius=0.30];
                \draw[thick] (Q2) circle[radius=0.30];
                \draw[thick] (Q3) circle[radius=0.30];
                \draw[thick] (Q4) circle[radius=0.30];
                \draw[thick] (Q5) circle[radius=0.30];

                \draw[-Latex,thick] (Q4) -- (Q1);
                \draw[-Latex,thick] (Q1) -- (Q5);
                \draw[-Latex,thick] (Q5) -- (Q2);
                \draw[-Latex,thick] (Q2) -- (Q4);
                \draw[-Latex,thick] (Q3) -- (Q4);
                \draw[-Latex,thick] (Q5) -- (Q3);
            \end{scope}
        \end{tikzpicture}
        \subcaption{A virtual switched system}
        \label{fig:virtual_system}
    \end{subfigure}

    \begin{subfigure}{1.0\linewidth}
        \centering
        \begin{tikzpicture}[scale=0.8]
            \draw[->] (0,0) -- (9,0) node[below] {$t$};
            \draw (0.8,0.12) -- (0.8,-0.12) node[below] {$t_0$};
            \draw (4.5,0.12) -- (4.5,-0.12) node[below] {$t_1$};
            \draw (8.2,0.12) -- (8.2,-0.12) node[below] {$t_2$};

            \begin{scope}[shift={(2.65,2.0)}]
                \node (P1) at (90:1.3) {1};
                \node (P2) at (162:1.3) {2};
                \node (P3) at (234:1.3) {3};
                \node (P4) at (306:1.3) {4};
                \node (P5) at (18:1.3) {5};

                \draw[thick] (P1) circle[radius=0.24];
                \draw[thick] (P2) circle[radius=0.36];
                \draw[thick] (P3) circle[radius=0.26];
                \draw[thick] (P4) circle[radius=0.40];
                \draw[thick] (P5) circle[radius=0.30];
            \end{scope}

            \begin{scope}[shift={(6.35,2.0)}]
                \node (Q1) at (90:1.3) {1};
                \node (Q2) at (162:1.3) {2};
                \node (Q3) at (234:1.3) {3};
                \node (Q4) at (306:1.3) {4};
                \node (Q5) at (18:1.3) {5};
                
                \draw[thick] (Q1) circle[radius=0.40];
                \draw[thick] (Q2) circle[radius=0.22];
                \draw[thick] (Q3) circle[radius=0.34];
                \draw[thick] (Q4) circle[radius=0.26];
                \draw[thick] (Q5) circle[radius=0.44];
            \end{scope}
        \end{tikzpicture}
        \subcaption{The GCS for \cref{fig:virtual_system}}
        \label{fig:virtual_gcs}
    \end{subfigure}

    \begin{subfigure}{1.0\linewidth}
        \centering
        \begin{tikzpicture}[scale=0.8]
            \draw[->] (0,0) -- (9,0) node[below] {$t$};
            \draw (0.8,0.12) -- (0.8,-0.12) node[below] {$t_0$};
            \draw (4.5,0.12) -- (4.5,-0.12) node[below] {$t_1$};
            \draw (8.2,0.12) -- (8.2,-0.12) node[below] {$t_2$};

            \begin{scope}[shift={(4.5,2.0)}]
                \node (P1) at (90:1.3) {1};
                \node (P2) at (162:1.3) {2};
                \node (P3) at (234:1.3) {3};
                \node (P4) at (306:1.3) {4};
                \node (P5) at (18:1.3) {5};

                \draw[thick] (P1) circle[radius=0.30];
                \draw[thick] (P2) circle[radius=0.42];
                \draw[thick] (P3) circle[radius=0.22];
                \draw[thick] (P4) circle[radius=0.38];
                \draw[thick] (P5) circle[radius=0.28];
            \end{scope}
        \end{tikzpicture}
        \subcaption{Our extended controllability score for \cref{fig:virtual_system}}
        \label{fig:virtual_cs}
    \end{subfigure}
    \caption{The difference between the GCS~\cite{Mo2025} and ours}
    \label{fig:difference}
\end{figure}

\Cref{fig:difference} illustrates this difference.
We consider a switched system as in \cref{fig:virtual_system}, where the structure switches at time $t_1$.
In this case, as depicted in \cref{fig:virtual_gcs}, where the radius indicates the importance of the node, the GCS assigns different scores to the networks over the interval $[t_0,t_1)$ and over $[t_1,t_2]$.
In contrast, as shown in \cref{fig:virtual_cs}, the controllability score we extend in this paper assigns the same score to each node over the entire time interval.

The second significant difference is the objective function.
If the volume of the reachable set or the average minimum energy is directly used as the objective function, as in \cite{SatoTerasaki2024}, \cref{prob:vcs} or \ref{prob:aecs} arises as discussed in \cref{subsec:settings}.
In contrast, since \cite{Mo2025} modifies the objective function, it deals with $m$ separate objective functions, each of which coincides with the volume of the reachable set or the average minimum energy for the LTI system on the corresponding subinterval.
Consequently, the method is essentially equivalent to computing the controllability score for the LTI system defined on each time interval, as explained in \cref{subsec:background}.

If the switching order is permuted, the GCS merely changes the ordering of the corresponding $p^{(k)}$ and may fail to detect any essential change in the dynamics.
In contrast, since the controllability Gramian~\cref{eq:p_gramian} incorporates the dynamics over the entire time interval, it can capture changes in dynamics that arise from permutations of the switching order.
Indeed, as shown in the numerical experiments in \cref{subsec:dependence}, our score can exhibit nontrivial changes when the switching order is altered.

\subsection{Uniqueness of controllability scores}
\label{subsec:uniqueness}

We prove that controllability scores are almost always unique.
The main idea is to modify the proof of \cite{SatoKawamura2025} to apply to systems considered here.
Let $D\coloneq\{\Delta t=(\Delta t_1,\ldots,\Delta t_m)\mid 0\leq \Delta t_k\leq s_k\ (k=1,\ldots,m)\}$ be the set of admissible tuples of durations.

The following well-known lemma~\cite{Mityagin2015} plays an essential role in the proof.
\begin{lemma}\label{prop:analytic_zero}
    Let $\varphi(\tau_1,\ldots,\tau_m)$ be a multivariate real-analytic function on some open set $E\subset\RR^m$.
    If $\varphi(\tau_1,\ldots,\tau_m)\not\equiv 0$, then $E'\coloneq\{(\tau_1,\ldots,\tau_m)\in E\mid \varphi(\tau_1,\ldots,\tau_m)\neq0\}$ is a dense open subset of $E$, and the Lebesgue measure of $E\setminus E'$ is $0$.
    Therefore, $\varphi(\tau_1,\ldots,\tau_m)\neq 0$ for almost all $(\tau_1,\ldots,\tau_m)\in E$ with respect to the Lebesgue measure.
\end{lemma}

Moreover, the following lemmas are also crucial.
\begin{lemma}\label{prop:uniqueness_lemma}
    Let $\Delta t\in D$ be fixed.
    If $R(\Delta t)$ is regular, then for $\Delta t$, \cref{prob:vcs,prob:aecs} each admit a unique optimal solution, where $R(\Delta t)\in\RR^{n\times n}$ and its $(i,j)$-th entry is given by
    \begin{equation}\label{eq:R_matrix}
        R(\Delta t)_{ij}\coloneq \int_0^{t_m} \left\{e_i^\top\Phi(t_m,\tau;\Delta t)e_j\right\}^2\rd\tau.
    \end{equation}
\end{lemma}
\begin{proof}
    See \cite[Theorem 1]{SatoKawamura2025}. \qed
\end{proof}

\begin{lemma}\label{prop:analyticity}
    All the entries of the matrix-valued function $R(\Delta t)$ are real analytic on $\mathrm{int}\,D$, where $\mathrm{int}\,D$ denotes the interior of $D$.
\end{lemma}
\begin{proof}
    See Appendix \ref{sec:proof_analyticity}. \qed
\end{proof}

The regularity of $R(\Delta t)$ is a sufficient condition that the objective functions of \cref{prob:vcs,prob:aecs} are strictly convex on the feasible region, which guarantees the uniqueness of the optimal solution.
From \cref{prop:uniqueness_lemma,prop:analyticity}, we can prove the following theorem.

\begin{theorem}\label{prop:uniqueness}
    Let
    \begin{multline*}
        \widetilde{D}_{\mathrm{V}}\coloneq\{\Delta t\in D\mid \text{For $\Delta t$, \Cref{prob:vcs}} \\
        \text{admits a unique optimal solution.}\},
    \end{multline*}
    and
    \begin{multline*}
        \widetilde{D}_{\mathrm{A}}\coloneq\{\Delta t\in D\mid \text{For $\Delta t$, \Cref{prob:aecs}} \\
        \text{admits a unique optimal solution.}\}.
    \end{multline*}
    Then, both $\widetilde{D}_{\mathrm{V}}$ and $\widetilde{D}_{\mathrm{A}}$ contain a dense open subset of $D$.
    Furthermore, both $D\setminus \widetilde{D}_{\mathrm{V}}$ and $D\setminus \widetilde{D}_{\mathrm{A}}$ have Lebesgue measure zero; therefore, \cref{prob:vcs,prob:aecs} each admit a unique optimal solution for almost all $\Delta t\in D$.
\end{theorem}

\begin{proof}
    By \cref{prop:analyticity}, $\varphi(\Delta t)\coloneq \det R(\Delta t) \ (\Delta t\in D)$ is real analytic on $\mathrm{int}\,D$.
    We will prove that $\varphi(\Delta t)\not\equiv0$ on $\mathrm{int}\,D$.
    Then, by letting $D'\coloneq \{\Delta t\in D\mid \varphi(\Delta t)\neq0\}$, it follows from \cref{prop:analytic_zero} that $D'\cap \mathrm{int}\,D$ is a dense open subset of $\mathrm{int}\,D$, and $\mathrm{int}\,D\setminus(D'\cap\mathrm{int}\,D)$ has Lebesgue measure zero.
    Thus, $D'$ is a dense open subset of $D$, and $D\setminus D'$ has Lebesgue measure zero.
    Furthermore, it follows from \cref{prop:uniqueness_lemma} that $D'\subset \widetilde{D}_{\mathrm{V}}$ and $D'\subset\widetilde{D}_{\mathrm{A}}$ hold, from which the claim holds.

    If follows from \cref{eq:transition} that $\Phi(t_m,\tau;\Delta t)=\Phi_1(\Delta t_1,\tau)$ for $0\leq\tau\leq\Delta t_1$ when $\Delta t=(\Delta t_1,0,\ldots,0)$.
    Thus,
    \begin{equation*}
        R(\Delta t_1,0,\ldots,0)_{ij}=\int_{0}^{\Delta t_1} \left\{e_i^\top\Phi_1(\Delta t_1,\tau)e_j\right\}^2\rd\tau
    \end{equation*}
    holds, and differentiating it with respect to $\Delta t_1$ yields
    \begin{multline*}
        \dfrac{\partial}{\partial \Delta t_1}R(\Delta t_1,0,\ldots,0)_{ij} \\
        =\{e_i^\top e_j\}^2+\int_{0}^{\Delta t_1} \dfrac{\partial}{\partial\Delta t_1}\left\{e_i^\top\Phi_1(\Delta t_1,\tau)e_j\right\}^2\rd\tau.
    \end{multline*}
    By substituting $\Delta t_1=0$, we obtain
    \begin{equation*}
        \left.\dfrac{\partial}{\partial \Delta t_1}R(\Delta t_1,0,\ldots,0)_{ij}\right\vert_{\Delta t_1=0}=e_i^\top e_j.
    \end{equation*}
    Therefore,
    \begin{equation*}
        R(\Delta t_1,0,\ldots,0)=(\Delta t_1)I+o(\Delta t_1)
    \end{equation*}
    holds, and hence $R(\Delta t_1,0,\ldots,0)$ is regular for sufficiently small $\Delta t_1>0$.
    This implies that $\varphi(\Delta t)\not\equiv0$ on $D$, and by continuity, $\varphi(\Delta t)\not\equiv0$ also on $\mathrm{int}\,D$, which completes the proof. \qed
\end{proof}

Note that ``for almost all $\Delta t\in D$'' in Theorem \ref{prop:uniqueness} cannot be replaced by ``all $\Delta t\in D$''.
In fact, when $m=1$, there exist a system matrix and a time parameter for which the optimal solution is not unique~\cite{SatoTerasaki2024}.

\subsection{Continuity of controllability scores}
\label{subsec:continuity}

We prove that the controllability scores are continuous with respect to $\Delta t$.
Let $p^*_{\mathrm{V}}(\Delta t)$ and $p^*_{\mathrm{A}}(\Delta t)$ denote the optimal solution to \cref{prob:vcs} for $\Delta t\in\widetilde{D}_{\mathrm{V}}$ and the optimal solution to \cref{prob:aecs} for $\Delta t\in\widetilde{D}_{\mathrm{A}}$, respectively.

Let $f\colon\RR^n\times\RR^m\to\RR\cup\{\infty\}$.
We say that $f$ is proper if there exists at least one pair $(p,\Delta t)$ such that $f(p,\Delta t)<\infty$.
It is lower semicontinuous if for every $\alpha\in\RR$, the sublevel set $\mathrm{lev}_{\leq\alpha}f\coloneq \{(p,\Delta t)\in\RR^n\times\RR^m\mid f(p,\Delta t)\leq\alpha\}$ is closed.
It is level-bounded in $p$ locally uniformly in $\Delta t$ if for every $\overline{\Delta t}$ and $\alpha\in\RR$, there exists a neighborhood $V\subset\RR^m$ of $\overline{\Delta t}$ along with a bounded set $B\subset\RR^n$ such that $\{p\in\RR^n\mid f(p,\Delta t)\leq\alpha\}\subset B$ for all $\Delta t\in V$.

The following proposition~\cite[Theorem~1.17]{RockafellarWets1998} is essential in the proof.
\begin{proposition}\label{prop:continuity_lemma}
    Let $f\colon\RR^n\times\RR^m\to\RR\cup\{\infty\}$ be proper, lower semicontinuous, and level-bounded in $p$ locally uniformly in $\Delta t$.
    Consider
    \begin{equation*}
        P(\Delta t)\coloneq\argmin_{p\in\RR^n}f(p,\Delta t),
    \end{equation*}
    where $P(\Delta t)=\emptyset$ when $f(p,\Delta t)=\infty$ for every $p\in\RR^n$.
    Assume that $p^{(k)}\in P(\Delta t^{(k)})$ and $\Delta t^{(k)}\to\overline{\Delta t}$, and that there exists $\overline{p}\in P(\overline{\Delta t})$ such that $f(\overline{p},\Delta t)$ is continuous in $\Delta t$ at $\overline{\Delta t}$.
    Then the sequence $\{p^{(k)}\}_{k\in\NN}$ is bounded, and all its cluster points lie in $P(\overline{\Delta t})$.
    In particular, when $P(\overline{\Delta t})=\{\overline{p}\}$, the sequence $\{p^{(k)}\}_{k\in\NN}$ converges to $\overline{p}$.
\end{proposition}

We can prove the following theorem.
\begin{theorem}\label{prop:continuity}
    The controllability scores $p^*_{\mathrm{V}}(\Delta t)$ and $p^*_{\mathrm{A}}(\Delta t)$ are continuous on $\widetilde{D}_{\mathrm{V}}$ and $\widetilde{D}_{\mathrm{A}}$, respectively.
\end{theorem}

\begin{proof}
    Let us define $g,h\colon\RR^n\times\RR^m\to\RR\cup\{\infty\}$ as
    \begin{align*}
        g(p,\Delta t)&\coloneq
        \begin{cases}
            -\log\det W(p;\Delta t) & \text{if} \quad
            \begin{aligned}
                \Delta t\in D,p\in\Delta, \\
                W(p;\Delta t)\succ O,
            \end{aligned} \\
            \infty & \text{otherwise},
        \end{cases} \\
        h(p,\Delta t)&\coloneq
        \begin{cases}
            \mathrm{tr}\left(W(p;\Delta t)^{-1}\right) & \text{if} \quad
            \begin{aligned}
                \Delta t\in D,p\in\Delta, \\
                W(p;\Delta t)\succ O,
            \end{aligned} \\
            \infty & \text{otherwise}.
        \end{cases}
    \end{align*}
    Then, $g$ and $h$ are proper.
    Moreover, $g$ is level-bounded in $p$ locally uniformly in $\Delta t$ since $\{p\in\RR^n\mid g(p,\Delta t)\leq\alpha\}\subset\Delta$ for all $\Delta t\in\RR^m$ and $\alpha\in\RR$, and $\Delta$ is bounded.
    Similarly, $h$ is also level-bounded in $p$ locally uniformly in $\Delta t$.

    Next, we prove that $g$ and $h$ are lower semicontinuous, i.e., $\mathrm{lev}_{\leq\alpha}g$ and $\mathrm{lev}_{\leq\alpha}h$ are closed for every $\alpha\in\RR$.
    Let $(p^{(k)},\Delta t^{(k)})\in\mathrm{lev}_{\leq\alpha}g$ with $(p^{(k)},\Delta t^{(k)})\to(\overline{p},\overline{\Delta t})\in\RR^n\times\RR^m$ as $k\to\infty$.
    Since $\det W(p^{(k)},\Delta t^{(k)})\geq\mathrm{e}^{-\alpha}$, by continuity we obtain $\det W(\overline{p},\overline{\Delta t})\geq\mathrm{e}^{-\alpha}$ and $W(\overline{p},\overline{\Delta t})\succ O$.
    Since $D$ and $\Delta$ are closed, we have $\overline{\Delta t}\in D$ and $\overline{p}\in\Delta$; hence $g(\overline{p},\overline{\Delta t})\leq\alpha$ implying $(\overline{p},\overline{\Delta t})\in\mathrm{lev}_{\leq\alpha}g$.
    This shows that $\mathrm{lev}_{\leq\alpha}g$ is closed.
    Similarly, we can prove that $\mathrm{lev}_{\leq\alpha}h$ is closed.

    Assume that $\Delta t^{(k)}\in\widetilde{D}_{\mathrm{V}}$ and $\Delta t^{(k)}\to\overline{\Delta t}\in\widetilde{D}_{\mathrm{V}}$ as $k\to\infty$.
    Then, $\argmin_{p\in\RR^n} g(p,\Delta t^{(k)})=\{p^*_{\mathrm{V}}(\Delta t^{(k)})\}$ and $\argmin_{p\in\RR^n} g(p,\overline{\Delta t})=\{p^*_{\mathrm{V}}(\overline{\Delta t})\}$, and $W(p^*_{\mathrm{V}}(\overline{\Delta t});\overline{\Delta t})\succ O$.
    By continuity, $W(p^*_{\mathrm{V}}(\overline{\Delta t});\Delta t)\succ O$ also holds for all $\Delta t$ sufficiently close to $\overline{\Delta t}$, and $g(p^*_{\mathrm{V}}(\overline{\Delta t}),\Delta t)=-\log\det W(p^*_{\mathrm{V}}(\overline{\Delta t});\Delta t)$ is continuous in $\Delta t$ at $\overline{\Delta t}$.
    Therefore, it follows from \cref{prop:continuity_lemma} that the sequence $\{p^*_{\mathrm{V}}(\Delta t^{(k)})\}_{k\in\NN}$ converges to $p^*_{\mathrm{V}}(\overline{\Delta t})$; consequently, $p^*_{\mathrm{V}}(\Delta t)$ is continuous on $\widetilde{D}_{\mathrm{V}}$.
    Similarly, it follows that $p^*_{\mathrm{A}}(\Delta t)$ is continuous on $\widetilde{D}_{\mathrm{A}}$. \qed
\end{proof}

\section{Algorithm and data-driven computation}
\label{sec:algorithm}

In this section, we discuss an algorithm to compute the controllability scores.
The algorithm consists of two primary parts.
The controllability Gramians are computed in the first step, and then the optimization algorithm is performed.
The optimization algorithm is identical to the one used in \cite{SatoTerasaki2024} and summarized in \cref{subsec:algorithm}.
The computation of the controllability Gramians requires knowledge of the system.
In \cref{subsec:model-based}, we elaborate on the computation in the case where the system is already identified.
However, system identification is challenging, especially for LTV systems; thus, the assumption that the system is already identified is not necessarily practical.
Therefore, we propose a data-driven method to compute the controllability Gramians in \cref{subsec:data-driven}.

Suppose that the tuple of durations is fixed throughout this section, and let us abbreviate $A(t;\Delta t)$ as $A(t)$, $\Phi(t,\tau;\Delta t)$ as $\Phi(t,\tau)$, and $W(p;\Delta t)$ as $W(p)$ for simplicity.
Thanks to \cref{prop:uniqueness}, assuming that controllability scores are uniquely determined does not cause a problem in most practical cases.

\subsection{Algorithm for controllability scores}
\label{subsec:algorithm}

\begin{algorithm}[t]
    \begin{algorithmic}[1]
        \REQUIRE The terminal condition $\varepsilon\geq0$ and the initial point $p^{(0)}\coloneq(1/n,\ldots,1/n)^\top$.
        \STATE Compute $W_i\ (i=1,\ldots,n)$ in \cref{eq:ltv_diagonal_gramian_base}.
        \FOR {$k=0,1,\ldots$}
            \STATE Determine a step size $\alpha^{(k)}$ by \cref{alg:armijo}.
            \STATE $p^{(k+1)}\leftarrow \Pi_{\Delta}\left(p^{(k)}-\alpha^{(k)}\nabla f(p^{(k)})\right)$
            \IF {$\norm{p^{(k)}-p^{(k+1)}}\leq\varepsilon$}
                \RETURN $p^{(k+1)}$
            \ENDIF
        \ENDFOR
    \end{algorithmic}
    \caption{The projected gradient method for \cref{prob:vcs,prob:aecs}}
    \label{alg:projected_gradient}
\end{algorithm}

\begin{algorithm}[t]
    \begin{algorithmic}[1]
        \REQUIRE The point $p^{(k)}$ in \cref{alg:projected_gradient}, the parameters $\sigma,\rho\in(0,1)$, and the initial step size $\alpha>0$.
        \STATE $\alpha^{(k)}\leftarrow\alpha$
        \WHILE {\TRUE}
            \STATE $\widetilde{p}^{(k)}\leftarrow \Pi_{\Delta}\left(p^{(k)}-\alpha^{(k)}\nabla f(p^{(k)})\right)$
            \IF {$f(\widetilde{p}^{(k)})\leq f(p^{(k)})+\sigma\nabla f(p^{(k)})^\top (\widetilde{p}^{(k)}-p^{(k)})$}
                \BREAK
            \ELSE
                \STATE $\alpha^{(k)}\leftarrow \rho\alpha^{(k)}$
            \ENDIF
        \ENDWHILE
        \RETURN $\alpha^{(k)}$
    \end{algorithmic}
    \caption{The Armijo rule}
    \label{alg:armijo}
\end{algorithm}

To calculate controllability scores, the optimal solutions to \cref{prob:vcs,prob:aecs}, we can employ the projected gradient method (\cref{alg:projected_gradient}), which is the same algorithm as the case of LTI systems~\cite{SatoTerasaki2024}.
Here, $f$ is the objective function, $\Pi_{\Delta}$ is a Euclidian projection onto $\Delta$, and $W_i$ is defined as
\begin{equation}\label{eq:ltv_diagonal_gramian_base}
    W_i\coloneq\int_0^{t_m} \Phi(t_m,\tau) e_ie_i^\top \Phi(t_m,\tau)^\top\rd\tau \quad (i=1,\ldots,n).
\end{equation}

In the first step of \cref{alg:projected_gradient}, we precompute the controllability Gramians $W_i\ (i=1,\ldots,n)$.
The knowledge of the system is utilized exclusively in the precomputation and is not required elsewhere.
We summarize the computation in the case where the system matrix $A(t)$ in \cref{eq:ltv_base} is known in \cref{subsec:model-based} and propose a data-driven method in \cref{subsec:data-driven}.
The only time-related parameter involved in \cref{alg:projected_gradient} is $\Delta t$.
Note that there is no need to specify the constant $s_k$, introduced in \cref{subsec:systems}, explicitly.

By employing $W_i \ (i=1,\ldots,n)$, we can calculate the controllability Gramian as
\begin{equation*}
    W(p)=\sum_{i=1}^n p_iW_i,
\end{equation*}
and the objective functions to \cref{prob:vcs,prob:aecs} can be computed:
\begin{gather*}
    g(p)\coloneq -\log\det W(p), \\
    h(p)\coloneq \mathrm{tr}\left(W(p)^{-1}\right).
\end{gather*}
Furthermore, their gradients can also be computed as
\begin{gather*}
    (\nabla g(p))_i=-\mathrm{tr}\left(W(p)^{-1}W_i\right), \\
    (\nabla h(p))_i=-\mathrm{tr}\left(W(p)^{-1}W_i W(p)^{-1}\right),
\end{gather*}
respectively.
Thus, after calculating $W_i$, the time complexity at each iteration of \cref{alg:projected_gradient} is $O(n^3)$.

The feasible region for \cref{prob:vcs,prob:aecs} is expressed as
\begin{equation*}
    \mathcal{F}\coloneq \left\{p\in\RR^n \;\middle|\; p\in\Delta,\ W(p)\succ O\right\}.
\end{equation*}
By setting the initial point $p^{(0)}=(1/n,\ldots,1/n)^\top$, we can guarantee that $p^{(0)}\in\mathcal{F}$.
The sequence $\{p^{(k)}\}$ generated by \cref{alg:projected_gradient} satisfies $f(p^{(k+1)})\leq f(p^{(k)})$~\cite{SatoTerasaki2024}, it follows by induction that $W(p^{(k)})\succ O$, and hence $p^{(k)}\in\mathcal{F}$.
Consequently, the constraints are automatically fulfilled if $p^{(0)}\in\mathcal{F}$.
The important points are that $\Pi_{\Delta}$ is employed in \cref{alg:projected_gradient} instead of the projection onto $\mathcal{F}$, and that $\Pi_{\Delta}$ can be efficiently computed~\cite{Condat2016}.
Thus, we can perform the projected gradient method efficiently.
Since the objective function is convex, as stated in \cref{prop:convexity}, the convergence to the global optimal solution is guaranteed~\cite{SatoTerasaki2024,Iusem2003}.
Furthermore, since the convexity is even strict in almost all cases, it is shown that the convergence rate is linear~\cite{SatoKawamura2025}; hence, the computation is guaranteed to be efficient.


\subsection{Model-based computation of controllability Gramians}
\label{subsec:model-based}

In this subsection, we elaborate on the model-based computation of the controllability Gramians $W_i \ (i=1,\ldots,n)$ in two cases.
The first case is where the system matrix $A(t)$ is general, and the second case is where the system is a temporal network, i.e., an LTV system \cref{eq:ltv_base} in which the matrix $A(t)$ is expressed by \cref{eq:ltv_matrix} and $A_k(t)$ is constant.
We assume that the system matrix $A(t)$ is already known in this subsection.

Since the matrix $\Phi(t_m,\tau)^\top$ is the inverse of the transpose of $\Phi(\tau,t_m)$, by \cref{eq:transition_definition}, it is the solution to
\begin{equation*}
    \dfrac{\partial}{\partial\tau}\Phi(t_m,\tau)^\top=-A(\tau)^\top\Phi(t_m,\tau)^\top,\quad \Phi(t_m,t_m)^\top=I.
\end{equation*}
Thus, we can naively calculate $W_i$~\cite{Perev2019} as
\begin{equation}\label{eq:W_integral}
    W_i=\int_{0}^{t_m} y_i(\tau)y_i(\tau)^\top\rd\tau,
\end{equation}
where $Y(t)=[y_1(t),\ldots,y_n(t)]^\top\in\RR^{n\times n}$ is the solution to
\begin{equation}\label{eq:W_equation}
    \dfrac{\rd}{\rd t}Y(t)=-A(t)^\top Y(t),\quad Y(t_m)=I.
\end{equation}

In numerical calculations, the differential equation and the integration are discretized with a specific time step size $\Delta\tau$.
The computational cost of solving \cref{eq:W_equation} is $O(n^3t_m/\Delta\tau)$ when $A(t)$ is dense, while it is $O(\mathrm{nnz}(A)nt_m/\Delta\tau)$ when $A(t)$ is a sparse matrix-valued function, where $\mathrm{nnz}(A)$ denotes the maximum number of nonzero entries of $A(t)$ with respect to $t$.

The overall computational cost of the following naive numerical quadrature  is $O(n^3t_m/\Delta\tau)$:
\begin{equation}\label{eq:W_naive}
    W_i\approx\sum_{\ell=0}^{t_m/\Delta\tau} y_i(\ell\Delta\tau)y_i(\ell\Delta\tau)^\top\Delta\tau \quad(i=1,\ldots,n).
\end{equation}
The computational cost is expensive; however, the use of orthogonal polynomial expansions has the potential to reduce it~\cite{Perev2019}.
Let
\begin{equation*}
    \varphi_{j}(t)\coloneq\sqrt{\dfrac{2j+1}{2}}\dfrac{1}{2^jj!}\dfrac{\rd^j}{\rd t^j}(t^2-1)^j\quad (j=0,1,2,\ldots)
\end{equation*}
be the normalized Legendre polynomials.
These polynomials form a complete orthonormal basis of $L^2([-1,1];\RR)$.
Then $W_i$ can be approximated as
\begin{equation}\label{eq:W_Legendre}
    W_i\approx\sum_{k=1}^{m}\sum_{j=0}^{J} q_{ij}^{(k)}\left(q_{ij}^{(k)}\right)^\top \quad (i=1,\ldots,n),
\end{equation}
where
\begin{equation}\label{eq:W_coefficient}
    q_{ij}^{(k)}=\sqrt{\dfrac{2}{\Delta t_k}}\int_{t_{k-1}}^{t_k} y_i(\tau)\varphi_j\left(\dfrac{2}{\Delta t_k}\tau-\dfrac{t_{k}+t_{k-1}}{\Delta t_k}\right)\rd\tau
\end{equation}
is the $j$-th Fourier vector coefficient of $y_i(t)$ on the interval $[t_{k-1},t_k]$, and $J$ is the order of series truncation in the Legendre orthogonal series approximation.
The minor difference from \cite{Perev2019} is that the interval $[0,t_m]$ is subdivided into $m$ subintervals.
While the expansion in Legendre polynomials exhibits rapid decay of the approximation error when the integrand is sufficiently smooth~\cite{WangXiang2012}, the convergence slows down in the absence of sufficient smoothness.
To improve the approximation accuracy, we therefore approximate the integrand separately on each subinterval so that the integrand becomes smooth within each piece.
When \cref{eq:W_coefficient} is approximated by numerical quadrature, the overall computational cost of \cref{eq:W_Legendre} is $O(n^2Jmt_m/\Delta\tau+n^3Jm)$.
Therefore, if the number of nodes $n$ is extremely large, whereas the number of abrupt changes $m$ remains small and an accurate approximation can be achieved with a relatively low polynomial degree $J$, then the overall computational complexity can be reduced.

Alternatively, for temporal networks, we can also use the Lyapunov equations.
Here, let $A_k(t)\equiv A_k\ (k=1,\ldots,m)$, and we make the following assumption.
\begin{assumption}\label{assump:eigenvalues}
    The matrices $A_k$ and $-A_k$ do not have a common eigenvalue for $k=1,\ldots,m$.
\end{assumption}
Then, we can use the following representation~\cite{Hou2023}:
\begin{subequations}
    \begin{gather}
        W_i=\sum_{k=1}^m E_k W_i^{(k)} E_k^\top, \\
        E_k=\mathrm{e}^{A_m\Delta t_m}\mathrm{e}^{A_{m-1}\Delta t_{m-1}}\ldots \mathrm{e}^{A_{k+1}\Delta t_{k+1}}, \\
        A_kW_i^{(k)}+W_i^{(k)}A_k^\top=-e_ie_i^\top+\mathrm{e}^{A_k\Delta t_k}e_ie_i^\top\mathrm{e}^{A_k^\top\Delta t_k}. \label{eq:lyapunov}
    \end{gather}
\end{subequations}
It follows from \cref{assump:eigenvalues} that \cref{eq:lyapunov} has a unique solution~\cite[Theorem 2.4.4.1]{HornJohnson2012}.
More specifically, the calculation is performed by \cref{alg:W_temporal}.
We can use the Bartels--Stewart algorithm~\cite{BartelsStewart1972} or the CF--ADI algorithm~\cite{LiWhite2002} to solve \cref{eq:lyapunov} in \cref{state:lyapunov}.
The choice of method should be determined by considering the time complexity and accuracy.

\begin{algorithm}[t]
    \begin{algorithmic}[1]
        \setcounter{ALC@unique}{0}
        \REQUIRE The adjacency matrices $A_1,\ldots,A_m$ and the durations $\Delta t_1,\ldots,\Delta t_m$
        \FOR {$i=1,\ldots,n$}
            \STATE $W_i\leftarrow O\in\RR^{n\times n}$
        \ENDFOR
        \STATE $E\leftarrow I\in\RR^{n\times n}$
        \FOR {$k=m,\ldots,1$}
            \FOR {$i=1,\ldots,n$}
                \STATE Compute $W_i^{(k)}$ by solving \cref{eq:lyapunov}. \label{state:lyapunov}
                \STATE $W_i\leftarrow W_i+EW_i^{(k)}E^\top$ \label{state:multiplication}
            \ENDFOR
            \STATE $E\leftarrow E\mathrm{e}^{A_k\Delta t_k}$
        \ENDFOR
        \ENSURE $W_i \ (i=1,\ldots,n)$
    \end{algorithmic}
    \caption{Model-based computation of $W_i$ for temporal networks}
    \label{alg:W_temporal}
\end{algorithm}

The Bartels--Stewart algorithm is a direct method that can be performed with the time complexity $O(n^3)$.
Therefore, the overall time complexity of \cref{alg:W_temporal} is $O(n^4m)$ when the Bartels--Stewart algorithm is used in \cref{state:lyapunov}.

The CF--ADI algorithm is an iterative method whose output is a low-rank approximation.
In the case of \cref{eq:lyapunov}, $W_i^{(k)}$ can be represented as $W_i^{(k)}=W_{i,1}^{(k)}-W_{i,2}^{(k)}$ where
\begin{subequations}\label{eq:lyapunov_cfadi}
    \begin{gather}
        A_kW_{i,1}^{(k)}+W_{i,1}^{(k)}A_k^\top=-e_ie_i^\top, \label{eq:lyapunov_cfadi1} \\
        A_kW_{i,2}^{(k)}+W_{i,2}^{(k)}A_k^\top=-\mathrm{e}^{A_k\Delta t_k}e_ie_i^\top\mathrm{e}^{A_k^\top\Delta t_k}. \label{eq:lyapunov_cfadi2}
    \end{gather}
\end{subequations}
By applying the CF--ADI algorithm to \cref{eq:lyapunov_cfadi1,eq:lyapunov_cfadi2}, we can obtain approximations $Z_{i,1}^{(k)}\in\RR^{n\times r_1}$ and $Z_{i,2}^{(k)}\in\RR^{n\times r_2}$ where $Z_{i,1}^{(k)}(Z_{i,1}^{(k)})^\top-Z_{i,2}^{(k)}(Z_{i,2}^{(k)})^\top \approx W_i^{(k)}$ and $r_1,r_2\ll n$.
The accuracy depends on the matrix $A_k$ and the iteration number.

When the CF--ADI algorithm is used in \cref{alg:W_temporal}, the time complexity excluding \cref{state:lyapunov} is $O(n^3rm)$, where $r$ is the maximum rank of the approximations.
Here, note that the matrix multiplication in \cref{state:multiplication} can be performed in $O(n^2r)$ thanks to the low-rank structure.
It is difficult to evaluate the time complexity of the CF--ADI algorithm since the total number of iterations depends on the tolerance of the approximation and the matrix $A_k$.

\begin{remark}\label{rem:cf_adi}
    Although $W_i^{(k)}=W_{i,1}^{(k)}-W_{i,2}^{(k)}$ is positive semidefinite if the computation is exact, the approximation $Z_{i,1}^{(k)}(Z_{i,1}^{(k)})^\top-Z_{i,2}^{(k)}(Z_{i,2}^{(k)})^\top (\approx W_i^{(k)})$ is not necessarily positive semidefinite.
    If $W_i$ is not positive semidefinite, \cref{alg:projected_gradient} might be unstable.
    Therefore, the tolerance of the CF--ADI algorithm must be small.
\end{remark}

\subsection{Data-driven computation of controllability Gramians}
\label{subsec:data-driven}

In this subsection, we propose a data-driven method to compute the controllability Gramians $W_i \ (i=1,\ldots,n)$.
We make the following assumption.
\begin{assumption}\label{assump:data-driven}
    The data of state trajectories $x^{(k)}(t),\ t\in[0,t_m] \ (k=1,\ldots,N)$ are given, and their initial values $x^{(1)}(0),\ldots,x^{(N)}(0)$ span $\RR^n$.
\end{assumption}

Under \cref{assump:data-driven}, $x^{(k)}(t) \ (k=1,\ldots,n)$ also span $\RR^n$ for all $t\in[0,t_m]$.
Thus, for $i=1,\ldots,n$, there exists $\alpha_i(t)\in\RR^{N}$ which satisfies
\begin{equation}\label{eq:alpha}
    X(t)\alpha_i(t)=e_i,
\end{equation}
where $X(t)\coloneqq [x^{(1)}(t),\ldots,x^{(N)}(t)]\in\RR^{n\times N}$.
We can represent $W_i$ as
\begin{align*}
    W_i&=\int_{0}^{t_m} \Phi(t_m,\tau) e_i e_i^\top \Phi(t_m,\tau)^\top \rd\tau \\
    &=\int_{0}^{t_m} \Phi(t_m,\tau)X(\tau)\alpha_i(\tau)\alpha_i(\tau)^\top X(\tau)^\top\Phi(t_m,\tau)^\top \rd\tau \\
    &=\int_{0}^{t_m} X(t_m)\alpha_i(\tau) \alpha_i(\tau)^\top X(t_m)^\top \rd\tau. \label{eq:W_representation}
\end{align*}
Thus, we can approximate $W_i$ as
\begin{equation}\label{eq:W_approximation}
    W_i \approx \sum_{\ell=0}^{t_m/\Delta\tau} X(t_m)\alpha_{i}(\ell\Delta\tau)\alpha_{i}(\ell\Delta\tau)^\top X(t_m)^\top \Delta\tau,
\end{equation}
where $\Delta\tau$ is the discretization width, and we do not require knowledge of $A(t)$ itself.
Although there may exist multiple solutions to \cref{eq:alpha}, we can use any of them.
The simplest choice is the least-norm solution, and we use it in \cref{subsec:performance}.

\begin{algorithm}[t]
    \begin{algorithmic}[1]
        \setcounter{ALC@unique}{0}
        \REQUIRE The $N$ observed data $X(t)\in\RR^{n\times N}$ and the discretization width $\Delta\tau$.
        \FOR {$i=1,\ldots,n$}
            \STATE $W_i\leftarrow O\in\RR^{n\times n}$
        \ENDFOR
        \FOR {$\ell=0,1,\ldots,t_m/\Delta\tau$}
            \STATE Perform a singular value decomposition for $X(\ell\Delta\tau)$. \label{state:svd}
            \FOR {$i=1,\ldots,n$}
                \STATE Compute the least-norm solution $\alpha_i(\ell\Delta\tau)$ to \cref{eq:alpha} with $t=\ell\Delta\tau$. \label{state:solve}
                \STATE $W_i\leftarrow W_i+X(t_m)\alpha_{i}(\ell\Delta\tau)\alpha_{i}(\ell\Delta\tau)^\top X(t_m)^\top \Delta\tau$ \label{state:sum}
            \ENDFOR
        \ENDFOR
        \ENSURE $W_i \ (i=1,\ldots,n)$
    \end{algorithmic}
    \caption{Data-driven approximation of $W_i$}
    \label{alg:W_approximation}
\end{algorithm}

More specifically, the approximation \cref{eq:W_approximation} is performed by \cref{alg:W_approximation}.
In \cref{state:svd}, we perform a singular value decomposition for $X(\ell\Delta\tau)\in\RR^{n\times N}$ in $O(n^2N)$.
The important point for our task is that \cref{eq:alpha} for each $i$ shares the same coefficient matrix $X(t)$.
Thus, once the decomposition is performed, we can compute the least-norm solution $\alpha_i(\ell\Delta\tau)$ in \cref{state:solve} in $O(nN)$.
\Cref{state:sum} can be computed in $O(nN)$.
Therefore, the time complexity of the overall procedure is $O(n^2Nt_m/\Delta\tau)$.


\begin{table*}[htbp]
    \centering
\begin{threeparttable}
    \footnotesize
    \caption{Methods to compute the controllability Gramians}
    \label{tab:complexity}
    \begin{center}
    \begin{tabular}{c||c|c|c|c|c}
        \hline
        & \multicolumn{4}{c|}{{\bf Model-based}}  & {\bf Data-driven} \\ \hline
        \multirow{2}{*}{Algorithm} & \cref{alg:W_temporal} & \cref{alg:W_temporal} & \multirow{2}{*}{\cref{eq:W_equation,eq:W_Legendre,eq:W_coefficient}} & \multirow{2}{*}{\cref{eq:W_equation,eq:W_naive}} & \multirow{2}{*}{\cref{alg:W_approximation}} \\
        & Bartels--Stewart  & CF--ADI & & \\ \hline
        Target & \multicolumn{2}{|c|}{Temporal networks} & LTV systems with \cref{eq:ltv_matrix} & \multicolumn{2}{c}{General LTV systems~\cref{eq:ltv_base}} \\ \hline
        \multirow{2}{*}{Constraint} & \multirow{2}{*}{\cref{assump:eigenvalues}} & \cref{assump:eigenvalues} & & & \multirow{2}{*}{\cref{assump:data-driven}} \\
        & & \cref{rem:cf_adi} & & \\ \hline
        Complexity & $O(n^4m)$ & $O(n^3rm)$\tnote{a} & $O((\mathrm{nnz}(A)n+n^2Jm)t_m/\Delta\tau+n^3Jm)$ & $O(n^3t_m/\Delta\tau)$ & $O(n^2Nt_m/\Delta\tau)$ \\ \hline
    \end{tabular}
    \begin{tablenotes}
        \item[a] The time complexity excludes \cref{state:lyapunov}.
    \end{tablenotes}
    \end{center}
\end{threeparttable}
\end{table*}

The methods to compute the controllability Gramians and their properties are summarized in \cref{tab:complexity}.
The four methods on the left are model-based ones discussed in \cref{subsec:model-based}, while the rightmost method is the data-driven one proposed in this section.
The two methods on the left are applicable only to temporal networks, whereas the method in the middle is applicable to LTV systems with \cref{eq:ltv_matrix}, and the two methods on the right are applicable to general LTV systems~\cref{eq:ltv_base}.
The methods for temporal networks require \cref{assump:eigenvalues}.
Furthermore, the tolerance of the CF--ADI algorithm must be small, as detailed in \cref{rem:cf_adi}.
The proposed data-driven method requires \cref{assump:data-driven}, which guarantees that the observed data span the entire space.

Which method is superior in terms of time complexity depends on the number of snapshots $m$, the maximum rank of the approximations $r$, the final time $t_m$, the order of series truncation in the Legendre orthogonal series approximation $J$, the maximum number of nonzero entries of $A(t)$ with respect to $t$, and the discretization width $\Delta\tau$.
In each model-based method, the controllability Gramians $W_i \ (i=1,\ldots,n)$ can be computed independently, allowing for parallel processing.
In \cref{alg:W_approximation}, they cannot be computed independently in the same way due to \cref{state:svd}.
However, since sequential computation with respect to time parameter $\ell$ is not necessary, we can perform parallel processing by offsetting $\ell$.
Therefore, when the number of processing units is less than or equal to $n$, the scalability is linear in both the model-based and data-driven methods.

\section{Numerical experiments}
\label{sec:experiments}

In this section, we show numerical examples.
First, we compare the controllability scores and the control energy centralities~\cite{Summers2016} for LTI and LTV systems on a directed network in \cref{subsec:directed_comparison}.
Next, we compare the controllability scores in this paper with the GCSs~\cite{Mo2025} for LTI and LTV systems on an undirected network in \cref{subsec:undirected_comparison}.
Subsequently, we investigate the numerical dependence of the controllability scores on the time parameters and the chronological order of the snapshots in \cref{subsec:dependence}.
Lastly, in \cref{subsec:performance}, we assess the performance of the data-driven method proposed in \cref{subsec:data-driven}.

Throughout this section, we disregard gradual changes for simplicity and employ simple examples of temporal networks consisting of the graphs shown in \cref{fig:temporal} and their aggregated networks shown in \cref{fig:aggregated}.
The temporal networks have four snapshots, and \cref{tab:temporal_networks} shows their orientation, chronological order, and the tuple of durations.
All the weights of the edges drawn in \cref{fig:temporal} are $1$.
Although each node in all the snapshots has a negative self-loop with a weight of $0.2$, we here omit to draw them for simplicity.
Let the adjacency matrices of (a), (b), (c), and (d) be denoted by $A_{\mathrm{a}}, A_{\mathrm{b}}, A_{\mathrm{c}}$, and $A_{\mathrm{d}}$, respectively.
Then, temporal network 6 is, for instance, represented by $A(t;\Delta t)$ in \cref{eq:ltv_matrix}, obtained by setting $A_1(t)\equiv A_{\mathrm{b}}, A_2(t)\equiv A_{\mathrm{d}}, A_3(t)\equiv A_{\mathrm{a}}, A_4(t)\equiv A_{\mathrm{c}}, \Delta t_1=2.1, \Delta t_2=1.8, \Delta t_3=1.9$, and $\Delta t_4=2.2$.
For all the networks, $t_4=8$.
Network 2 is an undirected network; in this case, we make the graphs in \cref{fig:temporal} undirected by adding the edge in the opposite direction (excluding self-loops); the adjacency matrices, thus, become symmetric.

The adjacency matrix of the aggregated network is defined as the weighted average of the snapshot adjacency matrices with respect to their durations.
Aggregated network 6 is, for instance, is an LTI system~\cref{eq:lti_base} represented by $A=\frac{2.1}{8}A_{\mathrm{b}}+\frac{1.8}{8}A_{\mathrm{d}}+\frac{1.9}{8}A_{\mathrm{a}}+\frac{2.2}{8}A_{\mathrm{c}}$.
Throughout all numerical experiments, we used the terminal condition parameter $\varepsilon=10^{-7}$ in \cref{alg:projected_gradient}.

\begin{figure}[htbp]
    \begin{tabular}{cc}
        \begin{minipage}[t]{0.45\linewidth}
            \centering
            \begin{tikzpicture}[scale=0.8]
                \node (n1) at (4.1,2.0) {1};
                \node (n2) at (0.7,2.5) {2};
                \node (n3) at (2.9,2.0) {3};
                \node (n4) at (1.7,2.0) {4};
                \node (n5) at (4.1,0.3) {5};
                \node (n6) at (1,0) {6};
                \node (n7) at (2.3,3.5) {7};
                \node (n8) at (2.9,0.3) {8};
                \node (n9) at (4.1,3.5) {9};
                \node (n10) at (0,1.2) {10};
                
                \draw [thick] (n1) circle [radius=0.28];
                \draw [thick] (n2) circle [radius=0.28];
                \draw [thick] (n3) circle [radius=0.28];
                \draw [thick] (n4) circle [radius=0.28];
                \draw [thick] (n5) circle [radius=0.28];
                \draw [thick] (n6) circle [radius=0.28];
                \draw [thick] (n7) circle [radius=0.28];
                \draw [thick] (n8) circle [radius=0.28];
                \draw [thick] (n9) circle [radius=0.28];
                \draw [thick] (n10) circle [radius=0.28];
    
                \draw [-Latex,thick] (n2)--(n10);
                \draw [-Latex,thick] (n3)--(n8);
                \draw [-Latex,thick] (n7)--(n2);
                \draw [-Latex,thick] (n7)--(n3);
                \draw [-Latex,thick] (n9)--(n1);
                \draw [-Latex,thick] (n10)--(n6);
            \end{tikzpicture}
            \subcaption{}
            \label{subfig:1st}
        \end{minipage} &
        \begin{minipage}[t]{0.45\linewidth}
            \centering
            \begin{tikzpicture}[scale=0.8]
                \node (n1) at (4.1,2.0) {1};
                \node (n2) at (0.7,2.5) {2};
                \node (n3) at (2.9,2.0) {3};
                \node (n4) at (1.7,2.0) {4};
                \node (n5) at (4.1,0.3) {5};
                \node (n6) at (1,0) {6};
                \node (n7) at (2.3,3.5) {7};
                \node (n8) at (2.9,0.3) {8};
                \node (n9) at (4.1,3.5) {9};
                \node (n10) at (0,1.2) {10};
                
                \draw [thick] (n1) circle [radius=0.28];
                \draw [thick] (n2) circle [radius=0.28];
                \draw [thick] (n3) circle [radius=0.28];
                \draw [thick] (n4) circle [radius=0.28];
                \draw [thick] (n5) circle [radius=0.28];
                \draw [thick] (n6) circle [radius=0.28];
                \draw [thick] (n7) circle [radius=0.28];
                \draw [thick] (n8) circle [radius=0.28];
                \draw [thick] (n9) circle [radius=0.28];
                \draw [thick] (n10) circle [radius=0.28];
    
                \draw [-Latex,thick] (n4)--(n6);
                \draw [-Latex,thick] (n7)--(n1);
                \draw [-Latex,thick] (n7)--(n4);
                \draw [-Latex,thick] (n9)--(n1);
            \end{tikzpicture}
            \subcaption{}
            \label{subfig:2nd}
        \end{minipage} \\

        \begin{minipage}[t]{0.45\linewidth}
            \centering
            \begin{tikzpicture}[scale=0.8]
                \node (n1) at (4.1,2.0) {1};
                \node (n2) at (0.7,2.5) {2};
                \node (n3) at (2.9,2.0) {3};
                \node (n4) at (1.7,2.0) {4};
                \node (n5) at (4.1,0.3) {5};
                \node (n6) at (1,0) {6};
                \node (n7) at (2.3,3.5) {7};
                \node (n8) at (2.9,0.3) {8};
                \node (n9) at (4.1,3.5) {9};
                \node (n10) at (0,1.2) {10};
                
                \draw [thick] (n1) circle [radius=0.28];
                \draw [thick] (n2) circle [radius=0.28];
                \draw [thick] (n3) circle [radius=0.28];
                \draw [thick] (n4) circle [radius=0.28];
                \draw [thick] (n5) circle [radius=0.28];
                \draw [thick] (n6) circle [radius=0.28];
                \draw [thick] (n7) circle [radius=0.28];
                \draw [thick] (n8) circle [radius=0.28];
                \draw [thick] (n9) circle [radius=0.28];
                \draw [thick] (n10) circle [radius=0.28];
    
                \draw [-Latex,thick] (n1)--(n5);
                \draw [-Latex,thick] (n2)--(n10);
                \draw [-Latex,thick] (n3)--(n8);
                \draw [-Latex,thick] (n7)--(n2);
                \draw [-Latex,thick] (n7)--(n3);
                \draw [-Latex,thick] (n10)--(n6);
            \end{tikzpicture}
            \subcaption{}
            \label{subfig:3rd}
        \end{minipage} &
        \begin{minipage}[t]{0.45\linewidth}
            \centering
            \begin{tikzpicture}[scale=0.8]
                \node (n1) at (4.1,2.0) {1};
                \node (n2) at (0.7,2.5) {2};
                \node (n3) at (2.9,2.0) {3};
                \node (n4) at (1.7,2.0) {4};
                \node (n5) at (4.1,0.3) {5};
                \node (n6) at (1,0) {6};
                \node (n7) at (2.3,3.5) {7};
                \node (n8) at (2.9,0.3) {8};
                \node (n9) at (4.1,3.5) {9};
                \node (n10) at (0,1.2) {10};
                
                \draw [thick] (n1) circle [radius=0.28];
                \draw [thick] (n2) circle [radius=0.28];
                \draw [thick] (n3) circle [radius=0.28];
                \draw [thick] (n4) circle [radius=0.28];
                \draw [thick] (n5) circle [radius=0.28];
                \draw [thick] (n6) circle [radius=0.28];
                \draw [thick] (n7) circle [radius=0.28];
                \draw [thick] (n8) circle [radius=0.28];
                \draw [thick] (n9) circle [radius=0.28];
                \draw [thick] (n10) circle [radius=0.28];
    
                \draw [-Latex,thick] (n1)--(n5);
                \draw [-Latex,thick] (n4)--(n6);
                \draw [-Latex,thick] (n7)--(n1);
                \draw [-Latex,thick] (n7)--(n4);
            \end{tikzpicture}
            \subcaption{}
            \label{subfig:4th}
        \end{minipage} \\
    \end{tabular}
    \caption{The snapshots of the temporal networks}
    \label{fig:temporal}
\end{figure}

\begin{figure}[htbp]
    \centering
    \begin{tikzpicture}[scale=0.8]
        \node (n1) at (4.1,2.0) {1};
        \node (n2) at (0.7,2.5) {2};
        \node (n3) at (2.9,2.0) {3};
        \node (n4) at (1.7,2.0) {4};
        \node (n5) at (4.1,0.3) {5};
        \node (n6) at (1,0) {6};
        \node (n7) at (2.3,3.5) {7};
        \node (n8) at (2.9,0.3) {8};
        \node (n9) at (4.1,3.5) {9};
        \node (n10) at (0,1.2) {10};
        
        \draw [thick] (n1) circle [radius=0.28];
        \draw [thick] (n2) circle [radius=0.28];
        \draw [thick] (n3) circle [radius=0.28];
        \draw [thick] (n4) circle [radius=0.28];
        \draw [thick] (n5) circle [radius=0.28];
        \draw [thick] (n6) circle [radius=0.28];
        \draw [thick] (n7) circle [radius=0.28];
        \draw [thick] (n8) circle [radius=0.28];
        \draw [thick] (n9) circle [radius=0.28];
        \draw [thick] (n10) circle [radius=0.28];

        \draw [-Latex,thick] (n1)--(n5);
        \draw [-Latex,thick] (n2)--(n10);
        \draw [-Latex,thick] (n3)--(n8);
        \draw [-Latex,thick] (n4)--(n6);
        \draw [-Latex,thick] (n7)--(n1);
        \draw [-Latex,thick] (n7)--(n2);
        \draw [-Latex,thick] (n7)--(n3);
        \draw [-Latex,thick] (n7)--(n4);
        \draw [-Latex,thick] (n9)--(n1);
        \draw [-Latex,thick] (n10)--(n6);
    \end{tikzpicture}
    \caption{The structure of the aggregated network}
    \label{fig:aggregated}
\end{figure}

\begin{table}[htbp]
    \centering
    \caption{Temporal networks employed for numerical experiments}
    \label{tab:temporal_networks}
    \begin{tabular}{c|ccc}
        \hline
        Network & Orientation & Order & Durations \\ \hline\hline
        1 & directed   & (a) $\rightarrow$ (b) $\rightarrow$ (c) $\rightarrow$ (d) & $(2.0,2.0,2.0,2.0)$ \\ 
        2 & undirected & (a) $\rightarrow$ (b) $\rightarrow$ (c) $\rightarrow$ (d) & $(2.0,2.0,2.0,2.0)$ \\
        3 & directed   & (a) $\rightarrow$ (b) $\rightarrow$ (c) $\rightarrow$ (d) & $(1.9,2.1,2.2,1.8)$ \\
        4 & directed   & (a) $\rightarrow$ (b) $\rightarrow$ (c) $\rightarrow$ (d) & $(0.5,2.8,1.8,2.9)$ \\
        5 & directed   & (b) $\rightarrow$ (d) $\rightarrow$ (a) $\rightarrow$ (c) & $(2.0,2.0,2.0,2.0)$ \\
        6 & directed   & (b) $\rightarrow$ (d) $\rightarrow$ (a) $\rightarrow$ (c) & $(2.1,1.8,1.9,2.2)$ \\
        7 & directed   & (b) $\rightarrow$ (d) $\rightarrow$ (a) $\rightarrow$ (c) & $(2.8,2.9,0.5,1.8)$ \\
        \hline
    \end{tabular}
\end{table}

\subsection{Comparison with existing centrality and between LTI and LTV systems in a directed network}
\label{subsec:directed_comparison}

In this subsection, we consider network 1 and present a comparison of controllability scores with existing centralities, the control energy centralities~\cite{Summers2016}.
There are three types of control energy centralities: the volumetric control energy (VCE) centrality, the average control energy (ACE) centrality, and the average controllability (AC) centrality.
They were originally proposed for LTI systems~\cref{eq:lti_base} but can be directly extended to LTV systems~\cref{eq:ltv_base}.
For each node $i$, they are defined in terms of the controllability Gramians~\cref{eq:ltv_diagonal_gramian_base} as follows:
\begin{align*}
    C_{\mathrm{VCE}}(i)&\coloneq\log\prod_{j=1}^{\mathrm{rank}\,W_i}\lambda_j(W_i), \\
    C_{\mathrm{ACE}}(i)&\coloneq-\mathrm{tr}\left(W_i^{\dagger}\right), \\
    C_{\mathrm{AC}}(i)&\coloneq\mathrm{tr}\left(W_i\right),
\end{align*}
where $\lambda_j(W_i)$ denotes the $j$-th largest eigenvalue and $W_i^{\dagger}$ denotes the pseudoinverse of $W_i$.
In all of these, a larger value indicates greater importance of the node.
We also compare them for the LTI and the LTV systems.

Network 1 has a hierarchical structure, and node 7 is an upstream node; therefore, it is expected to have high importance.
In contrast, nodes 5, 6, and 8 are downstream nodes and are thus expected to have low importance.

\begin{table}[htbp]
    \centering
    \caption{Comparison between controllability scores and control energy centralities}
    \label{tab:directed_comparison}
    \begin{subtable}[t]{1.0\linewidth}
        \centering
        \caption{Centralities for temporal network 1}
        \label{tab:temporal_directed}
        \begin{tabular}{c|ccccc}
            \hline
            Node & VCS & AECS & VCE & ACE & AC \\ \hline\hline
            1 & 0.058       & 0.154       & 2.371       & -1.458       & 15.614 \\
            2 & 0.142       & 0.105       & 0.628       & -7.357       & 9.361 \\ 
            3 & 0.150       & 0.154       & 1.551       & -1.111       & 5.243 \\
            4 & 0.107       & 0.136       & 1.804       & -1.767       & 10.729 \\
            5 & 0.000       & 0.000       & 0.875       &\textbf{-0.417}& 2.398 \\
            6 & 0.000       & 0.000       & 0.875       &\textbf{-0.417}& 2.398 \\
            7 &\textbf{0.341}&\textbf{0.232}& 0.498     & -49.568      & \textbf{92.147} \\
            8 & 0.000       & 0.000       & 0.875       &\textbf{-0.417}& 2.398 \\
            9 & 0.167       & 0.115       & \textbf{3.904}& -0.517     & 25.635 \\
            10& 0.034       & 0.105       & 1.551       & -1.111       & 5.243 \\ \hline
        \end{tabular}
    \end{subtable}

    \begin{subtable}[t]{1.0\linewidth}
        \centering
        \caption{Centralities for aggregated network 1}
        \label{tab:aggregated_directed}
        \begin{tabular}{c|ccccc}
            \hline
            Node & VCS & AECS & VCE & ACE & AC \\ \hline\hline
            1 & 0.077       & 0.168       &\textbf{1.591}& -1.475       & 7.243 \\
            2 & 0.165       & 0.115       & 1.022       & -8.399       & 15.277 \\ 
            3 & 0.163       &\textbf{0.177}&\textbf{1.591}& -1.475       & 7.243 \\
            4 & 0.117       & 0.117       &\textbf{1.591}& -1.475       & 7.243 \\
            5 & 0.000       & 0.000       & 0.875       &\textbf{-0.417}& 2.398 \\
            6 & 0.000       & 0.000       & 0.875       &\textbf{-0.417}& 2.398 \\
            7 &\textbf{0.249}& 0.165      & 0.778       & -62.563      &\textbf{78.507} \\
            8 & 0.000       & 0.000       & 0.875       &\textbf{-0.417}& 2.398 \\
            9 & 0.192       & 0.120       & 1.022       & -8.399       & 15.277 \\
            10& 0.036       & 0.139       &\textbf{1.591}& -1.475       & 7.243 \\ \hline
        \end{tabular}
    \end{subtable}
\end{table}

\Cref{tab:directed_comparison} shows the controllability scores and the control energy centralities for temporal network 1 (\cref{tab:temporal_directed}) and for aggregated network 1 (\cref{tab:aggregated_directed}).
First, we examine the results for the temporal network.
The VCS and the AECS yield results consistent with the expectation: they assign the highest importance to node 7 and assign no importance to nodes 5, 6, and 8.
The AC produces a similar outcome, giving node 7 the highest importance while ranking nodes 5, 6, and 8 at the bottom.
By contrast, the VCE and the ACE exhibit a different pattern, placing node 7 at the lowest rank while ACE assesses nodes 5, 6, and 8 as most important.
Although the result may seem unusual, it is because the definition of the VCE and the ACE does not necessarily reflect controllability, as explained in \cite{SatoTerasaki2024}.

Next, we examine the results for the aggregated network.
The overall tendency is similar to that for the temporal network.
A noteworthy difference, however, is that while the AECS assigned the highest importance to node 7 in the temporal network, it ranks node 3 highest in the aggregated network.
In other words, increasing the temporal resolution and considering the chronological order of snapshots might change the most important node in the AECS.
Therefore, for networks with time-varying structures, approximating them using an LTI system on the aggregated network is insufficient for understanding their dynamics.
Instead, modeling them as an LTV system on a temporal network is considered to allow for a more precise understanding.

The VCS tends to assign higher scores to upstream nodes than the AECS.
Since upstream nodes in hierarchical networks are obviously important, the AECS is expected to reveal critical nodes whose significance is not apparent solely from the network topology, as explained in \cite{SatoTerasaki2024}.

\subsection{Comparison with existing centrality and between LTI and LTV systems in an undirected network}
\label{subsec:undirected_comparison}

In this subsection, we consider network 2 and present a comparison of the controllability scores with the GCSs~\cite{Mo2025}.
Unlike network 1, network 2 becomes an undirected network by adding edges in the opposite direction.
As a result, it does not possess a simple hierarchical structure, making it difficult to infer node importance solely from the network structure.

\Cref{tab:vcs_undirected} shows the VCS and the generalized volumetric controllability score (GVCS) for network 2, and \cref{tab:aecs_undirected} shows the AECS and the generalized average energy controllability score (GAECS) for network 2.
Here, temp. means the values for temporal network 2, while agg. means the values for aggregated network 2.
Since the GVCS or the GAECS computes the VCS or the AECS for each snapshot, \cref{tab:vcs_undirected,tab:aecs_undirected} presents the VCS and the AECS corresponding to snapshots (a), (b), (c), and (d) in \cref{fig:temporal}.
Furthermore, for the aggregated network, the GVCS and the GAECS coincide with the VCS and the AECS, respectively.

\begin{table}[htbp]
    \centering
    \caption{Comparison between controllability scores and GCSs~\cite{Mo2025}}
    \label{tab:undirected_comparison}
    \begin{subtable}[t]{1.0\linewidth}
        \centering
        \caption{VCS and GVCS for network 2}
        \label{tab:vcs_undirected}
        \begin{tabular}{c|c|cccc|c}
            \hline
            \multirow{2}{*}{Node} & VCS & \multicolumn{4}{c|}{GVCS} & VCS \\
            & temp. & (a) & (b) & (c) & (d) & agg. \\ \hline\hline
            1 &\textbf{0.150}& 0.100      & 0.100       & 0.100       & 0.100       & 0.100 \\
            2 & 0.111       & 0.100       & 0.100       & 0.100       & 0.100       & 0.100 \\
            3 & 0.107       & 0.100       & 0.100       & 0.100       & 0.100       & 0.100 \\
            4 & 0.079       & 0.100       & 0.100       & 0.100       & 0.100       & 0.100 \\
            5 & 0.000       & 0.100       & 0.100       & 0.100       & 0.100       & 0.100 \\
            6 & 0.108       & 0.100       & 0.100       & 0.100       & 0.100       & 0.100 \\
            7 & 0.111       & 0.100       & 0.100       & 0.100       & 0.100       & 0.100 \\
            8 & 0.094       & 0.100       & 0.100       & 0.100       & 0.100       & 0.100 \\
            9 &\textbf{0.136}& 0.100      & 0.100       & 0.100       & 0.100       & 0.100 \\
            10& 0.103       & 0.100       & 0.100       & 0.100       & 0.100       & 0.100 \\ \hline
        \end{tabular}
    \end{subtable}

    \begin{subtable}[t]{1.0\linewidth}
        \centering
        \caption{AECS and GAECS for network 2}
        \label{tab:aecs_undirected}
        \begin{tabular}{c|c|cccc|c}
            \hline
            \multirow{2}{*}{Node} & AECS & \multicolumn{4}{c|}{GAECS} & AECS \\
            & temp. & (a) & (b) & (c) & (d) & agg. \\ \hline\hline
            1 &\textbf{0.244}& 0.100      &\textbf{0.143}& 0.100      &\textbf{0.143}&\textbf{0.145} \\
            2 & 0.064       & 0.108       & 0.084       & 0.108       & 0.084       & 0.080 \\
            3 & 0.056       &\textbf{0.128}& 0.084      &\textbf{0.128}& 0.084      & 0.120 \\
            4 &\textbf{0.221}& 0.077      &\textbf{0.143}& 0.077      &\textbf{0.143}& 0.080 \\
            5 & 0.056       & 0.077       & 0.084       & 0.100       & 0.093       & 0.071 \\
            6 & 0.075       & 0.086       & 0.093       & 0.086       & 0.093       & 0.111 \\
            7 & 0.103       & 0.108       & 0.109       & 0.108       & 0.100       &\textbf{0.141} \\
            8 & 0.078       & 0.086       & 0.084       & 0.086       & 0.084       & 0.071 \\
            9 & 0.052       & 0.100       & 0.093       & 0.077       & 0.084       & 0.071 \\
            10& 0.052       &\textbf{0.128}& 0.084      &\textbf{0.128}& 0.084      & 0.111 \\ \hline
        \end{tabular}
    \end{subtable}
\end{table}

First, we examine the results for the VCS and the GVCS.
A remarkable point is that the GVCS assigns uniform evaluations to all nodes in each snapshot.
Similarly, the VCS also provides uniform scores to all nodes for the aggregated network.
In contrast, the VCS for the temporal network yields nontrivial results.

This result can be explained by Theorem 2 in \cite{SatoKawamura2025}.
Specifically, when the system matrix $A$ of an LTI system~\cref{eq:lti_base} is symmetric, the VCS gives uniform scores to all nodes.
Since the system matrix $A$ defined by an undirected graph is symmetric, the GVCS evaluates all the nodes uniformly in each snapshot.
Similarly, the aggregated network yields an LTI system over the entire time interval whose system matrix $A$ is symmetric, and hence the VCS also assigns uniform scores to all nodes.
In contrast, the temporal network lies outside the scope of Theorem 2 in \cite{SatoKawamura2025}; therefore, even if the system matrix $A$ is symmetric in each snapshot, the VCS may produce nontrivial results.
Indeed, for temporal network 2, the VCS produces a nontrivial result.
While the VCS fails to serve as a meaningful measure for undirected networks, the findings suggest that, in the case of LTV networks, the VCS can indeed produce informative values.

Next, we examine the results for the AECS and the GAECS.
The same observations as \cref{subsec:directed_comparison} apply to the difference between the AECS for the temporal and aggregated networks.
Specifically, while node 4 is ranked as the second most important in the temporal network, node 7 holds this position in the aggregated network.
This demonstrates that increasing temporal resolution and incorporating chronological order can change the evaluation given by the AECS.
Therefore, for networks with time-varying structures, it is more accurate to model them as LTV systems on temporal networks rather than as LTI systems on aggregated networks.

For snapshots (a) and (c), the GAECS assigns high scores to nodes 3 and 10.
In contrast, the AECS on the temporal network gives these nodes low scores, showing that incorporating the chronological order of snapshots can significantly alter the evaluation.
Furthermore, since nodes 3 and 10 are perfectly symmetric in each individual snapshot, the GAECS assigns them identical scores.
However, once the chronological order is considered, they are no longer symmetric, and the AECS assigns different scores.
This suggests that the AECS, applied to temporal networks, provides a more appropriate assessment of node importance.

For both the VCS and the AECS, the relationship between the network structure and the score remains an open problem even in the case of LTI systems.
Investigating this issue is left for future work.

\subsection{Dependence of controllability scores on time parameters and chronological order of snapshots}
\label{subsec:dependence}

In this subsection, we consider networks 1, 3, 4, 5, 6, and 7.
We examine how the scores change when the time parameters or the chronological order of the snapshots are modified.
All of these networks have hierarchical structures, and as discussed in \cref{subsec:directed_comparison}, node 7 is expected to have high importance, whereas nodes 5, 6, and 8 are expected to have low importance.

\Cref{tab:vcs_dependence} shows the VCS for each network, and \cref{tab:aecs_dependence} shows the AECS.
Temporal network 3 is obtained from temporal network 1 by making a small perturbation of the time parameters, while temporal network 4 is obtained by a larger perturbation.
As established in \cref{prop:continuity}, the controllability scores vary continuously with respect to the time parameters.
Accordingly, the difference between the controllability scores of temporal networks 1 and 3 is expected to be relatively small, while that between networks 1 and 4 might be relatively large, which is indeed confirmed by the results.
A similar relation holds for temporal networks 5, 6, and 7, and the results are consistent with this expectation.

On the other hand, when the chronological order of the snapshots is altered, the controllability score can change significantly.
In fact, in temporal networks 1, 3, and 4, node 9 is assigned the second-highest VCS value, whereas in temporal networks 5, 6, and 7, which are obtained by permuting the snapshot order together with the corresponding time parameters, node 2 becomes the second most important in terms of the VCS.
This suggests that reordering the snapshots can substantially alter the dynamics and that the controllability scores are capable of capturing such changes, whereas, as noted in \cref{subsec:difference}, the GCSs~\cite{Mo2025} merely permute the scores and do not necessarily detect the change in dynamics.

\begin{table}[htbp]
    \centering
    \caption{Dependence of controllability scores on time parameters and chronological order of snapshots}
    \label{tab:dependence}
    \begin{subtable}[t]{1.0\linewidth}
        \centering
        \caption{VCS for networks 1, 3, 4, 5, 6, 7}
        \label{tab:vcs_dependence}
        \begin{tabular}{c|cccccc}
            \hline
            \multirow{2}{*}{Node} & \multicolumn{6}{c}{Network} \\
            & 1 & 3 & 4 & 5 & 6 & 7 \\ \hline\hline
            1 & 0.058       & 0.064       & 0.064       & 0.102       & 0.110       & 0.120 \\
            2 & 0.142       & 0.145       & 0.124       &\textbf{0.185}&\textbf{0.191}&\textbf{0.159} \\
            3 & 0.150       & 0.149       & 0.140       & 0.118       & 0.115       & 0.157 \\
            4 & 0.107       & 0.098       & 0.112       & 0.094       & 0.090       & 0.102 \\
            5 & 0.000       & 0.000       & 0.000       & 0.000       & 0.000       & 0.000 \\
            6 & 0.000       & 0.000       & 0.000       & 0.000       & 0.000       & 0.000 \\
            7 &\textbf{0.341}&\textbf{0.346}&\textbf{0.319}&\textbf{0.335}&\textbf{0.334}&\textbf{0.324} \\
            8 & 0.000       & 0.000       & 0.032       & 0.000       & 0.000       & 0.046 \\
            9 &\textbf{0.167}&\textbf{0.165}&\textbf{0.154}& 0.161       & 0.157       & 0.138 \\
            10& 0.034       & 0.034       & 0.055       & 0.003       & 0.003       & 0.000 \\ \hline
        \end{tabular}
    \end{subtable}

    \begin{subtable}[t]{1.0\linewidth}
        \centering
        \caption{AECS for networks 1, 3, 4, 5, 6, 7}
        \label{tab:aecs_dependence}
        \begin{tabular}{c|cccccc}
            \hline
            \multirow{2}{*}{Node} & \multicolumn{6}{c}{Network} \\
            & 1 & 3 & 4 & 5 & 6 & 7 \\ \hline\hline
            1 &\textbf{0.154} &\textbf{0.159}& 0.133    &\textbf{0.172}&\textbf{0.176}&\textbf{0.155} \\
            2 & 0.105       & 0.103       & 0.111       & 0.135       & 0.136       & 0.144 \\
            3 &\textbf{0.154}& 0.153       & 0.123      & 0.164       & 0.162       & 0.154 \\
            4 & 0.136       & 0.129       &\textbf{0.156}       & 0.086       & 0.085       & 0.101 \\
            5 & 0.000       & 0.000       & 0.000       & 0.000       & 0.000       & 0.000 \\
            6 & 0.000       & 0.000       & 0.000       & 0.000       & 0.000       & 0.000 \\
            7 &\textbf{0.232}&\textbf{0.229}&\textbf{0.226}&\textbf{0.196}&\textbf{0.196}&\textbf{0.193} \\
            8 & 0.000       & 0.000       & 0.058       & 0.000       & 0.000       & 0.064 \\
            9 & 0.115       & 0.114       & 0.114       & 0.107       & 0.106       & 0.105 \\
            10& 0.105       & 0.113       & 0.079       & 0.140       & 0.140       & 0.085 \\ \hline
        \end{tabular}
    \end{subtable}
\end{table}

\subsection{Performance of data-driven method}
\label{subsec:performance}

In this subsection, we examine the performance of the proposed data-driven method to compute the controllability Gramians~\cref{eq:ltv_diagonal_gramian_base}, especially the relation between accuracy and the number of observed data.
The number of the nodes is $n=10$.
We vary the number of observed data, $N=7, 8, 9, 10, 11, 12$, to examine how the accuracy changes.
In particular, the cases $N=7,8,9$ do not satisfy \cref{assump:data-driven} and thus are expected to fail to produce correct results; nevertheless, they are included to assess performance when the number of observations is small.

We generated the observed data $x^{(k)}(t)$ as follows:
\begin{description}
    \item [Step 1.] The initial state $x^{(k)}(0)$ is randomly generated by the uniform distribution over the unit sphere, i.e., $\norm{x^{(k)}(0)}=1$.
    \item [Step 2.] The state trajectory $x^{(k)}(t)$ is obtained accurately.
    \item [Step 3.] The state trajectory $x^{(k)}(\ell\Delta\tau) \ (\ell=1,\ldots,T/\Delta\tau)$ is observed with the sampling period $\Delta\tau=10^{-3}$.
\end{description}
By Step 1., \cref{assump:data-driven} is satisfied with probability $1$ in the case $N=10,11,12$.
By using observed data, we computed the controllability Gramians.

\begin{figure}[htbp]
    \centering
    \begin{tikzpicture}
        \begin{axis}[
            boxplot/draw direction=y,
            ymode=log,
            every boxplot/.style={
                draw=black,
                fill=none,
            },
            every boxplot median/.style={draw=black, very thick},
            every boxplot whisker/.style={draw=black},
            every boxplot box/.style={draw=black},
            ylabel={Relative error},
            xlabel={The number of the observed data},
            xtick={1,2,3,4,5,6},
            xticklabels={$7$,$8$,$9$,$10$,$11$,$12$}
            ]
            
            \addplot[boxplot] table[y index=0] {gramian_error.dat};
            \addplot[boxplot] table[y index=1] {gramian_error.dat};
            \addplot[boxplot] table[y index=2] {gramian_error.dat};
            \addplot[boxplot] table[y index=3] {gramian_error.dat};
            \addplot[boxplot] table[y index=4] {gramian_error.dat};
            \addplot[boxplot] table[y index=5] {gramian_error.dat};
        \end{axis}
    \end{tikzpicture}
    \caption{The error of the proposed data-driven method}
    \label{fig:accuracy}
\end{figure}

We conducted the experiment 100 times for each value of $N$.
In each experiment, the error was measured as the maximum relative error:
\begin{equation*}
    \max_{i=1,\ldots,n}\dfrac{\norm{\widetilde{W_i}-W_i}}{\norm{W_i}},
\end{equation*}
where $W_i$ is the exact controllability Gramian~\cref{eq:ltv_diagonal_gramian_base}, and $\widetilde{W_i}$ is the computed one by \cref{alg:W_approximation}.
The result is summarized in the boxplot shown in \cref{fig:accuracy}.
When $N<n$, \cref{assump:data-driven} is violated, leading to significant errors; however, when $N\geq n$, the computation achieves high accuracy.

\section{Concluding remarks}
\label{sec:conclusion}

\subsection{Summary}

We have extended the controllability score to apply to LTV systems, which include dynamical systems on temporal networks.
We have also proved the uniqueness of the controllability score for almost all time parameter values and its continuity with respect to the time parameters.
From the result, we consider that the controllability score is unique in most practical cases.
Furthermore, we have proposed a data-driven method to compute controllability scores for practical use and compared it with model-based methods.

In the numerical experiments, we demonstrated that the controllability score may yield different results for temporal and aggregated networks.
Consequently, when studying systems that are more naturally modeled as LTV systems rather than LTI systems, the proposed extension is essentially significant.

We also compared the controllability score with the existing centralities, namely the control energy centrality~\cite{Summers2016} and the GCS~\cite{Mo2025}.
Whereas the VCE and the ACE contradicted the importance expected from the network structure, the controllability score and the AC provided scores consistent with the structure.
Moreover, the GCS essentially ignores the effect of the chronological order of snapshots on the dynamics, whereas the controllability score can capture this effect.
In addition, while the GVCS produces meaningless scores for undirected temporal networks, the VCS can, in some cases, yield meaningful ones.

Finally, we evaluated the performance of the proposed data-driven method.
The results show that the accuracy is very poor when \cref{assump:data-driven} is violated, but once the assumption is satisfied, the method achieves highly accurate computations.
Hence, the proposed method is expected to allow us to assess network centrality using experimental data rather than knowledge of the system matrix.

\subsection{Relationship between network structure and controllability scores: open question}

In \cref{subsec:directed_comparison}, we observed that in hierarchical networks, the controllability scores tend to assign higher scores to upstream nodes than downstream nodes.
Nevertheless, in general, the relationship between network structure and the controllability score has not yet been systematically understood, even in the LTI setting.
Clarifying this relationship and extending the analysis to LTV systems are left for future work.

\subsection{Application to real data}

While the controllability score has been applied to brain networks in \cite{SatoKawamura2025}, its applications to real data are still limited.
Applying the controllability score to large-scale real systems and analyzing such systems from the viewpoint of the controllability score to elucidate their network functions are left for future work.

\appendices

\section{Proof of \cref{prop:analyticity}}
\label{sec:proof_analyticity}

The following proposition is a well-known result on ordinary differential equations~\cite{Chicone2024}.
\begin{proposition}\label{prop:transition_regularity}
    Assume that $A'(t)$ is real analytic on $(0,s)$ and continuous on $[0,s]$.
    Then, the corresponding state transition matrix $\Phi'(t,\tau)$ is real analytic with respect to $t$ on $(0,s)$ for each fixed $\tau$.
\end{proposition}

\begin{proof}[\cref{prop:analyticity}]
    By definition, we obtain
    \begin{align*}
        R(\Delta t)_{ij}&=\int_{0}^{t_m} \left\{e_i^\top\Phi(t_m,\tau;\Delta t)e_j\right\}^2 \rd\tau \\
        &=\sum_{k=1}^{m} \int_{t_{k-1}}^{t_{k}} \left\{e_i^\top\Phi(t_m,\tau;\Delta t)e_j\right\}^2 \rd\tau \\
        &=\sum_{k=1}^{m} \int_{0}^{\Delta t_k} \left\{e_i^\top\Phi(t_m,\tau+t_{k-1};\Delta t)e_j\right\}^2 \rd\tau.
    \end{align*}
    It follows from \cref{eq:transition} that
    \begin{equation*}
        \Phi(t_m,\tau+t_{k-1};\Delta t)=\Phi_m(\Delta t_m,0)\cdots\Phi_k(\Delta t_k,0)\Phi_k(\tau,0)^{-1}
    \end{equation*}
    for $0\leq \tau\leq\Delta t_k \ (k=1,\ldots,m)$.
    From \cref{prop:transition_regularity}, $\Phi_k(\Delta t_k,0)$ is real analytic with respect to $\Delta t_k$ on $(0,s_k)$; hence $\Phi_m(\Delta t_m,0)\cdots\Phi_k(\Delta t_k,0)$ is real analytic with respect to $\Delta t$ on $\mathrm{int}\,D$, and $\Phi_k(\tau,0)^{-1}$ is real analytic with respect to $\tau$ on $(0,\Delta t_k)$.
    Thus, we can represent the integrand as
    \begin{equation*}
        \left\{e_i^\top\Phi(t_m,\tau+t_{k-1};\Delta t)e_j\right\}^2=\sum_{\ell=1}^{L_k}f^{(k)}_{\ell}(\Delta t)g^{(k)}_{\ell}(\tau),
    \end{equation*}
    where $f^{(k)}_{\ell}(\Delta t)$ is real analytic with respect to $\Delta t$ on $\mathrm{int}\,D$, $g^{(k)}_{\ell}(\tau)$ is real analytic with respect to $\tau$ on $(0,s_k)$, and $L_k$ is the number of terms.
    
    Therefore, it follows that
    \begin{equation*}
        R(\Delta t)_{ij}=\sum_{k=1}^{m}\sum_{\ell=1}^{L_k} f_{\ell}^{(k)}(\Delta t)\int_{0}^{\Delta t_k}g_{\ell}^{(k)}(\tau)\rd\tau
    \end{equation*}
    holds.
    Since $\int_{0}^{\Delta t_k}g_{\ell}^{(k)}(\tau)\rd\tau$ is real analytic with respect to $\Delta t_k$, it is therefore real analytic with respect to $\Delta t$ as well, which completes the proof. \qed
\end{proof}

\section*{Acknowledgment}
This work was supported by the Japan Society for the Promotion of Science KAKENHI under Grant 23K03899. 

\bibliographystyle{IEEEtran}
\bibliography{reference.bib}

\end{document}